\documentclass[graybox]{svmult}

\usepackage{mathptmx}       
\usepackage{helvet}         
\usepackage{courier}        
\usepackage{type1cm}        

\usepackage{makeidx}         
\usepackage{graphicx}        
\usepackage{multicol}        
\usepackage[bottom]{footmisc}
\usepackage{amsmath,amsfonts,amssymb}
\usepackage{latexsym}
\usepackage{epsfig,overpic}

\usepackage{pgfplots}
\usetikzlibrary{shapes,arrows,snakes,calendar,matrix,backgrounds,folding,calc,positioning,spy}
\usepackage{tikz}

\usepackage{todonotes}

\makeindex             

\newcommand{\mT}{\mathcal T}
\newcommand{\mF}{\mathcal F}
\newcommand{\mV}{\mathcal V}
\newcommand{\id}{{\rm id}}

\newcommand{\Vreg}{\mathcal{V}_{\text{reg},h}}

\newcommand{\rr}{\mathbb{R}}

\newcommand{\nn}{\mathbb{N}}

\newcommand{\jumpleft}{[\hspace*{-0.05cm}[}
\newcommand{\jumpright}{]\hspace*{-0.05cm}]}
\newcommand{\jump}[1]{\jumpleft #1 \jumpright}
\newcommand{\spacejump}[1]{\jump{#1}}

\newcommand{\hphi}{\hat \phi_h}
\newcommand{\lin}{{\text{\tiny lin}}}
\newcommand{\Omegalin}{\Omega^{ \lin}}
\newcommand{\Gammalin}{\Gamma^{ \lin}}

\newcommand{\thetah}{\Theta_h}

\begin{document}

\title*{A Higher Order Isoparametric Fictitious Domain Method for Level Set Domains}
\titlerunning{A Higher Order Isoparametric Fictitious Domain Method}
\author{Christoph Lehrenfeld}
\institute{Christoph Lehrenfeld \at Institut f\"ur Numerische und Angewandte Mathematik, Lotzestr. 16-18, D-37083 G\"ottingen, Germany, \email{lehrenfeld@math.uni-goettingen.de}}
%
%
\maketitle

\abstract{
We consider a new fictitious domain approach of higher order accuracy. To implement Dirichlet conditions we apply the classical Nitsche method combined with a facet-based stabilization (ghost penalty). Both techniques are combined with a higher order isoparametric finite element space which is based on a special mesh transformation. The mesh transformation is build upon a higher order accurate level set representation and allows to reduce the problem of numerical integration to problems on domains which are described by piecewise linear level set functions. The combination of this strategy for the numerical integration and the stabilized Nitsche formulation results in an accurate and robust method. We introduce and analyze it and give numerical examples.
}

\abstract*{
We consider a new fictitious domain approach of higher order accuracy. To implement Dirichlet conditions we apply the classical Nitsche method combined with a facet-based stabilization (ghost penalty). Both techniques are combined with a higher order isoparametric finite element space which is based on a special mesh transformation. The mesh transformation is build upon a higher order accurate level set representation and allows to reduce the problem of numerical integration to problems on domains which are described by piecewise linear level set functions. The combination of this strategy for the numerical integration and the stabilized Nitsche formulation results in an accurate and robust method. We introduce and analyze it and give numerical examples.
}


\section{Introduction}
\label{sec:1}
\subsection{Motivation}
In physics, biology, chemistry and engineering many applications of simulation science involve complex geometrical shapes. In the past decade research on methods which separate the geometry description from the computational mesh and in turn provide a more flexible handling of the geometry compared to traditional conforming mesh descriptions have become very popular.
Significant progress has been made in the recent years concerning the construction, analysis and application of \emph{fictitious domain} finite element methods, often also called \emph{unfitted FEM}, see for instance the papers \cite{
burman2014cutfem,
burman2012fictitious,
fries2010extended,
grossreusken07,
hansbo2002unfitted,
olshanskii2009finite} 
and the references therein.
 In the literature also other names are used for this class of discretization methods, e.g.  extended FEM (XFEM) and CutFEM. 
 While most of the work on unfitted discretizations has been on piecewise linear (unfitted) finite elements, many unfitted discretizations have a natural extension to higher order finite element spaces, see for instance
 \cite{bastian2009unfitted,johanssonhigh2013,massjung12,parvizianduesterrank07}.
 However, new techniques are needed for higher order accuracy as
 challenges with respect to geometrical accuracy, stability (small cuts) and conditioning arise for higher order discretizations.
 
In this contribution, we consider a standard model problem which contains the essential numerical challenges (geometry handling, stability, conditioning) and propose a higher order discretization based on a ghost penalty stabilization which handles stability and conditioning issues in the situation of small cuts and an isoparametric mapping which facilitates the higher order accurate numerical treatment of the geometry. 

\subsection{The model problem}
As a model problem we consider the Poisson problem on an open and bounded domain $\Omega \subset \rr^d,~d=2,3$,
  \begin{equation} \label{eq:ellmodel}
    - \Delta u = \, f ~\text{ in }~ \Omega, \quad 
    u = \, u_D ~\text{ on }~ \Gamma := \partial \Omega, \quad u_D\in L^2(\partial \Omega).
  \end{equation}
We discuss the assumptions on the smoothness of the domain boundary in more detail below.
A well-posed weak formulation of \eqref{eq:ellmodel} is: \\
Find $u \in H^1_D(\Omega) := \{u \in H^1 | \operatorname{tr}_\Gamma(u) = u_D \}$ such that
\begin{equation}
  a(u,v) := \int_{\Omega} \nabla u \nabla v \, dx = \int_{\Omega} f \, v \, dx =: f(v), \quad \text{ for all } v \in H^1_0(\Omega).
\end{equation}
We want to solve the problem in a so-called \emph{unfitted} setting, i.e. in a setting where $\Omega$ is not meshed exactly but only implicitly represented by a scalar level set function $\phi$, $\Omega := \{ \phi \leq 0 \}$. This flexibility gives rise to several problems, the treatment of which this contribution is devoted to. 

\subsection{Literature}
In the original paper by Nitsche \cite{Nitsche71} a variational formulation has been proposed to implement Dirichlet boundary conditions weakly without imposing boundary conditions as \emph{essential} conditions into the finite element space. Such an approach is also often used in \emph{unfitted} finite element methods where the imposition of boundary conditions as \emph{essential} conditions into the space is hardly possible. A variant of Nitsche's method has been applied to unfitted interface problems in the seminal paper \cite{hansbo2002unfitted}. In that paper an averaging operator of the flux at the interface is tailored to provide stability independent of the cut position at the interface. For unfitted boundary value problems the corresponding stabilising mechanism is missing and additional stabilizations had to be invented to obtain robust methods, cf. e.g. \cite{BeckerBurmanHansbo2010,wadbro2013uniformly}. Recently, a popular stabilization is the ghost penalty method \cite{Burman2010} which has been successfully applied to (among other problems) unfitted boundary value problems, see for instance \cite{burman2012fictitious,MassingLarsonLoggEtAl2013a}. 
With Nitsche's method and stabilizations as the ghost penalty method, stable discretizations have been derived for a variety of PDE problems \cite{burman2016cutbernoulli,burman2016shape,hansbo2014cut,MassingLarsonLoggEtAl2013a,massing2016stabilized}.
However, when aiming at higher order discretizations the proper treatment of geometries still represents a significant difficulty. This is due to the fact that higher order accurate and robust numerical integration on domains which are implicitly described by level set functions is very challenging.

Different ideas exist in the literature to approach this problem. A very basic approach is based on a piecewise linear approximation of the geometry through a piecewise linear approximation of the level set function, cf. section \ref{pcwlin}. As this approach is inherently limited to second order accuracy it is insufficient for higher order methods.

Constructing quadrature points and weights based on the idea of fitting certain integral moments has been considered in \cite{muller2013highly,sudhakar2013quadrature}, with the drawback that stability of the resulting quadrature rule can not be guaranteed in general.
A different approach has been considered in \cite{cheng2010higher,dreau2010studied,fries2015} where a subtriangulation of the mesh is combined with a parametric mapping.
An implementation of this approach is technical and ensuring robustness is difficult, especially in more than two dimensions. 

In \cite{burman2015cut} an improvement of the basic piecewise linear approximation has been suggested based on the idea of boundary value corrections where the imposition of the boundary values in Nitsche's method are adjusted to the distance of the discrete approximation of the boundary to the (implicitly known) exact boundary. The domain of integration that is required in the method only needs to be a second order approximation. In the very recent paper \cite{boiveau2016fictitious} a variant of this method has been investigated for the nonsymmetric penalty-free Nitsche method.

In the following we consider another method which is also based on a piecewise linear approximation of the geometry which is then combined with a parametric mapping of the underlying mesh. This method has been proposed in \cite{lehrenfeld15} and applied and analysed for scalar interface problems \cite{LR16a,LR16b}, Stokes interface problems \cite{LPWL_PAMM_2016} and PDEs on surfaces \cite{GLR16}. We complement this series of studies by an application of the method for an unfitted boundary value problem as for instance in \cite{boiveau2016fictitious,burman2015cut}.

The major contributions in this study are the discussion of the isoparametric concept for unfitted finite element methods, the introduction of an isoparametric fictitious domain method and its thorough a priori error analysis. The method and its analysis is new. While for the analysis of geometrical errors we make use of results recently obtained in \cite{lehrenfeld15,LR16a,LR16b,GLR16}, the analysis of the ghost penalty stabilization for isoparametric (unfitted) finite elements has not been addressed in the literature so far.


\subsection{Structure of the paper}
In section \ref{sec:prelim} we start with preliminaries, introduce notation, assumptions and the basic structure of the geometry handling. The construction of the parametric mapping used in the isoparametric method is explained in more detail in section \ref{sec:mapping} where also the most important properties of the mapping are given. In section \ref{sec:isoparammethod} we then define the isoparametric finite element method for the discretization of \eqref{eq:ellmodel}. The main part of this contribution is the analysis  of the method in section 
\ref{sec:analysis} in which we make use of the results in \cite{LR16a,LR16b,GLR16} at many places. Nevertheless, we require new results for stability, (geometrical) consistency and approximation due to the ghost penalty terms that have not been used in the methods in \cite{LR16a,LR16b,GLR16}.
We validate the a priori bounds with numerical experiments in section \ref{sec:numex} before we conclude this contribution in section \ref{sec:outlook}.

\section{Preliminaries} \label{sec:prelim}
We introduce basic assumptions on the geometry representation and notation for domains and triangulations.
\subsection{Assumptions on the domain description}
We assume that the domain $\Omega$ is embedded into a larger polygonal domain $\widetilde{\Omega} \supset \Omega$. The boundary $\Gamma = \partial \Omega$ is described by a level set function $\phi: \widetilde{\Omega} \rightarrow \rr$ on $\widetilde{\Omega}$ so that $\Gamma = \{ x \in \tilde{\Omega}, \phi(x) = 0\}$ and $\phi$ is negative in $\Omega$. In a neighborhood $U_\Gamma$ of $\Gamma$ the level set function is assumed to be smooth, $\phi \in C^{m+1}(U_\Gamma)$ for a $m \in \nn,~ m > 1$.
The level set function is not necessarily a signed distance function but there exist constants $c,C>0$ (independent of $\Gamma$) such that there holds $ c |\phi(x)| \leq \operatorname{dist}(x,\Gamma) \leq C |\phi(x)|$, $x\in U_\Gamma$.

Here and in the following we will use the notation $\lesssim$ ($\gtrsim$) for inequalities with generic constants $c$ which are independent of the mesh size $h$ and independent of the position of the domain boundary $\Gamma$ relative to the mesh, $a \leq c b \Leftrightarrow a \lesssim b$. If there holds $a \lesssim b$ and $a \gtrsim b$, we also write $a \simeq b$.

We consider a simplicial shape regular triangulation $\widetilde{\mT}_h$ of the domain $\widetilde{\Omega}$ and assume that a good approximation $\phi_h \in V_h^k$ of $\phi$ is known,
where $V_h^k$ is the space of continuous elementwise polynomials of (at most) degree $k$.
This means that we assume 
\begin{equation*}
  \Vert \phi - \phi_h \Vert_{\infty,U_\Gamma} + h \Vert \nabla \phi - \nabla \phi_h \Vert_{\infty,U_\Gamma} \lesssim h^{k+1}
\end{equation*}
where $\Vert \cdot \Vert_{\infty,U_\Gamma}$ denotes the $L^{\infty}(U_\Gamma)$ norm.
We assume that the smoothness of $\phi$ in $U_\Gamma$ has $m \geq k$.

Finally, we make the assumption of the mesh $\mT_h$ being quasi-uniform, i.e. there exists $h$ such that $h \simeq h_T$ for all $T \in \mT_h$ with $h_T=\operatorname{diam}(T)$. While only the (local) shape regularity is crucial for the techniques applied in the construction of the method and its analysis, this assumption is chosen to simplify the presentation. In section \ref{sec:numex} we include a numerical test case with a mesh that is not quasi-uniform. 

\subsection{A piecewise linear approximation of the geometry} \label{pcwlin}
As a basis for the geometry approximation of the only implicitly defined boundary $\Gamma$, we use a piecewise linear approximation $\hphi \in V_h^1$ of $\phi_h$. $\hphi$ defines an approximation $\Gammalin := \{ x \in \tilde{\Omega}, \hphi(x) = 0\}$ of the domain boundary which is planar on each (simplicial) element and hence allows for an \emph{explicit} representation. 
This explicit representation facilitates the numerical integration that is required in the weak formulation of unfitted finite element discretizations.
Obtaining quadrature rules based on this geometry representation is discussed in (among others) \cite[Chapter 5]{naerland2014geometrychap5},\cite{mayer2009interface} for triangles and tetrahedra and in \cite{lehrenfeld2015nitsche,lehrenfeld2015diss} also for 4-prisms and pentatopes (4-simplices).
Many simulation codes which apply unfitted finite element discretizations, e.g. \cite{DROPS,engwer2012dune,burman2014cutfem,renard2014getfem++,carraro2015implementation} make use of this kind of strategy.
The obvious drawback of this approach is that it is (by construction) only second order accurate.

\subsection{Improved geometrical accuracy with a parametric mapping}
In order not to suffer from the low order approximation we apply a special parametric transformation $\Theta_h \in (V_h^k)^d$ on the mesh $\widetilde{\mathcal{T}}_h$.
The mapping has been introduced in \cite{lehrenfeld15} and is discussed in more detail below in section \ref{sec:mapping}. For now, we assume that the transformation is given and defines a bijection on $\widetilde\Omega$.
From the \emph{explicitly} known mesh transformation $\Theta_h \in (V_h^k)^d$ and 
the (low order) geometry approximation with an \emph{explicit} representation $\Gammalin$ and $\Omegalin := \{ x \in \tilde{\Omega}, \hphi(x) \leq 0 \}$ we obtain a new approximation of $\Gamma$ and $\Omega$ as
\begin{equation}
  \Gamma_h := \{ \Theta_h(x), x \in \Gammalin \}~~\text{and}~~\Omega_h:=\{\Theta_h(x), x \in \Omegalin\}.
\end{equation}
  \begin{figure}
    \begin{center}
      \includegraphics[width=0.7\textwidth]{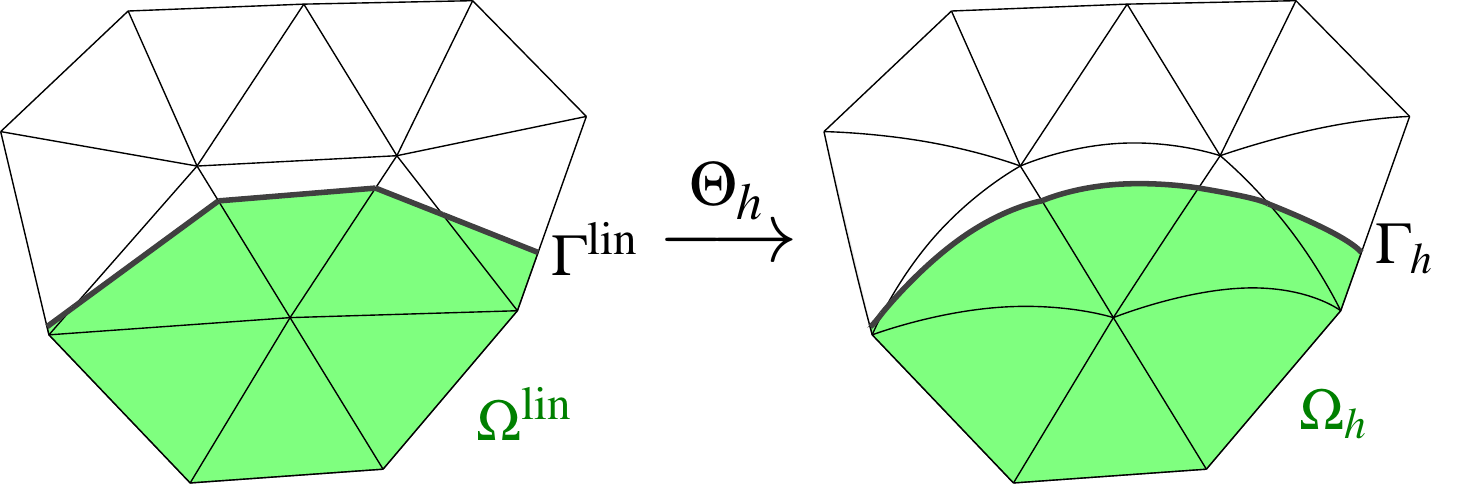} 
    \end{center}
    \caption{Application of the mesh transformation to improve the geometry approximation.}
    \label{fig:trafo}
  \end{figure}
A sketch of the application of the mesh transformation is given in Fig. \ref{fig:trafo}.
  
We note that $\Omega_h$ and $\Gamma_h$ have \emph{explicit parametrizations} which is crucial to obtain robust numerical integration strategies.

\subsection{Notation for cut elements}
We introduce notation corresponding to the cut configuration in the mesh. We note that the cut topology does not change under transformation with $\Theta_h$ so that it depends only on the piecewise linear approximation $\Omegalin$ ($\Gammalin$).

We define the ``active'' part of the background mesh $\widetilde{\mT}_h$ as \vspace*{-0.1cm}
\begin{subequations}
\begin{align}
  \mT_h & := \{ T \in \widetilde{\mT}_{h} : T \cap \Omegalin \neq \emptyset \}.  \\
\intertext{Cut elements are gathered in the subset} 
  \mT_h^\Gamma & := \{T \in \mT_h, T \cap \Gammalin \neq \emptyset \}, \\
\intertext{and the extension by direct (through edges) neighbors in}
  \mT_h^{\Gamma,+} & := \{T \in \mT_h, \operatorname{meas}_{1} (\overline{T}\cap\overline{T'}) > 0, T' \in \mT_h^{\Gamma} \}, \\
\intertext{ where $\operatorname{meas}_m$ denotes the $m$-dimensional Hausdorff-measure. 
Element interfaces between two elements in $\mT_h^{\Gamma,+}$ are collected in }
\mF_h & := \{ F = \overline{T^+} \cap \overline{T^-}: T^+, T^- \in \mT_h^{\Gamma,+}, \operatorname{meas}_{d-1}(F) > 0 \}.
\end{align}
\end{subequations}
A sketch of the domains and their (sub-)triangulations is given in Fig. \ref{fig:mesh}.
\begin{figure}
  \vspace*{-0.25cm}
  \begin{center}
    \includegraphics[width=0.975\textwidth]{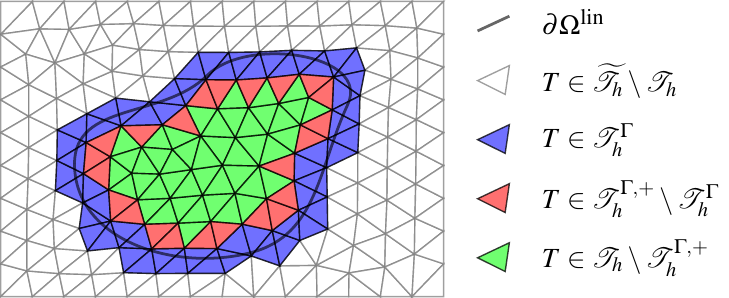} 
  \end{center}
  \vspace*{-0.175cm}
  \caption{Simplex triangulation of a domain $\widetilde{\Omega}$ and the different set of elements corresponding to the cut configuration.}
  \label{fig:mesh}
  \vspace*{-0.25cm}
\end{figure}
Finally, $\Omega^\Gamma := \{ x \in T, T \in \mT_h^\Gamma \}$ denotes the domain of cut elements and we define $\Omega^\mT = \{ x \in T, T \in \mT_h \}$ the domain of elements which have some part in $\Omegalin$. We use a corresponding definition for $\Omega^{\Gamma,+}$. In the following we assume that $h$ is sufficiently small so that $\Omega^{\Gamma} \subset U_\Gamma$, i.e. $\phi$ is smooth in $\Omega^{\Gamma}$.
\section{A Parametric mesh transformation for higher order accurate geometry approximation} \label{sec:mapping}
We introduce the mesh transformation $\Theta_h$ that is used later on for the higher order fictitious domain finite element method. 
\subsection{Construction of the mapping} \label{sec:mapping:construction}
The goal of the mesh transformation is to achieve a mapping which has $\Theta_h(\Gammalin) \approx \Gamma$ and is a homeomorphic finite element (vector) function in $(V_h^k)^d$. The mesh transformation should further be the identity in the larger part of the domain.
In \cite{lehrenfeld15} we developed such a transformation.
The basic idea is to characterize a locally ideal transformation $\Psi$ with a one-to-one mapping.
To a point $x \in T \in \mT_h^\Gamma$ we find a suitable point $y \in \widetilde \Omega$ such that 
\begin{equation}\label{eq:psi1}
  \hphi(x) = \phi (y)
\end{equation}
and define $\Psi(x) := y$. 
For a point $x$ and a corresponding approximated level set value $c = \hphi(x)$ there may be infinitely many points $y$ with $\phi(y) = c$. Hence, we specifiy the search direction $G(x) = \nabla \phi / \Vert \nabla \phi \Vert$ and 
ask for
\begin{equation} \label{eq:psi}
\Psi(x) := y = x + d(x) \cdot G(x)
\end{equation}
where $d(x)$ is the smallest (in absolute value) number such that \eqref{eq:psi1} is true.
As $\phi$ is typically not known, we make a first approximation by replacing $\phi$ with $\phi_h$ in \eqref{eq:psi1} and $G$ with $G_h := \nabla \phi_h / \Vert \nabla \phi_h \Vert$. The thusly defined mapping still gives rise to problems.
For a point $x \in T$ a corresponding mapped point $y$ could be positioned in a different element $T'$. In view of computational complexity this is undesired as it requires non-local operations (evaluation of $\phi_h|_{T'}$) which can be costly especially in parallel environments. We circumvent this by replacing $\phi_h$ with $\mathcal{E}_T \phi_h$, the polynomial extension of $\phi_h|_T$ to $\rr^d$. For a correspondingly adapted transformation $\Psi_h$ there holds for $x \in T \in \mT_h^\Gamma$
\begin{subequations}
\begin{equation} \label{psih}
  \Psi_h(x) = x + d_h(x) \cdot G_h(x)
\end{equation}
where $d_h(x)$ is the smallest (in absolute value) number so that 
\begin{equation} \label{trafo}
\hphi(x) = \mathcal{E}_T \phi_h (\Psi_h(x))
\end{equation}
\end{subequations}
holds. The pointwise evaluation of $\Psi_h$ for $x \in T \in \mT_h^\Gamma$ can be realized efficiently, cf. \cite{lehrenfeld15}.
The transformation $\Psi_h$ is only elementwise smooth, $\Psi_h|_T \in C^\infty(T), ~ T \in \mT_h^\Gamma$ but can be discontinuous across element interfaces. However, the jumps across element interfaces are of higher order so that a suitable projection into the finite element space $(V_h^k)^d$ of continuous vector-valued functions allows to remove the discontinuities with introducing only a higher order error.
To achieve this we apply a projection $P_h^1: C(\mT_h^\Gamma) \rightarrow V_h^k|_{\Omega^\Gamma}$ which maps (component-wise) an only piecewise continuous function onto $V_h^k|_{\Omega^\Gamma}$, the space of continous functions on $\Omega^\Gamma$ which are elementwise polynomials of degree (at most) $k$ on $\mT_h^\Gamma$.
A second projection $P_h^2$ realizes (component-wise) the finite element extension from $V_h^k|_{\Omega^\Gamma}$ to $V_h^k$.
The composition of both projections $P_h = P_h^2 P_h^1$ allows for the definition of the finite element mesh transformation:
\begin{equation} \label{eq:thetah}
  \Theta_h := P_h \Psi_h.
\end{equation}

\begin{figure}[h!]
  \begin{center}
    \includegraphics[width=0.975\textwidth]{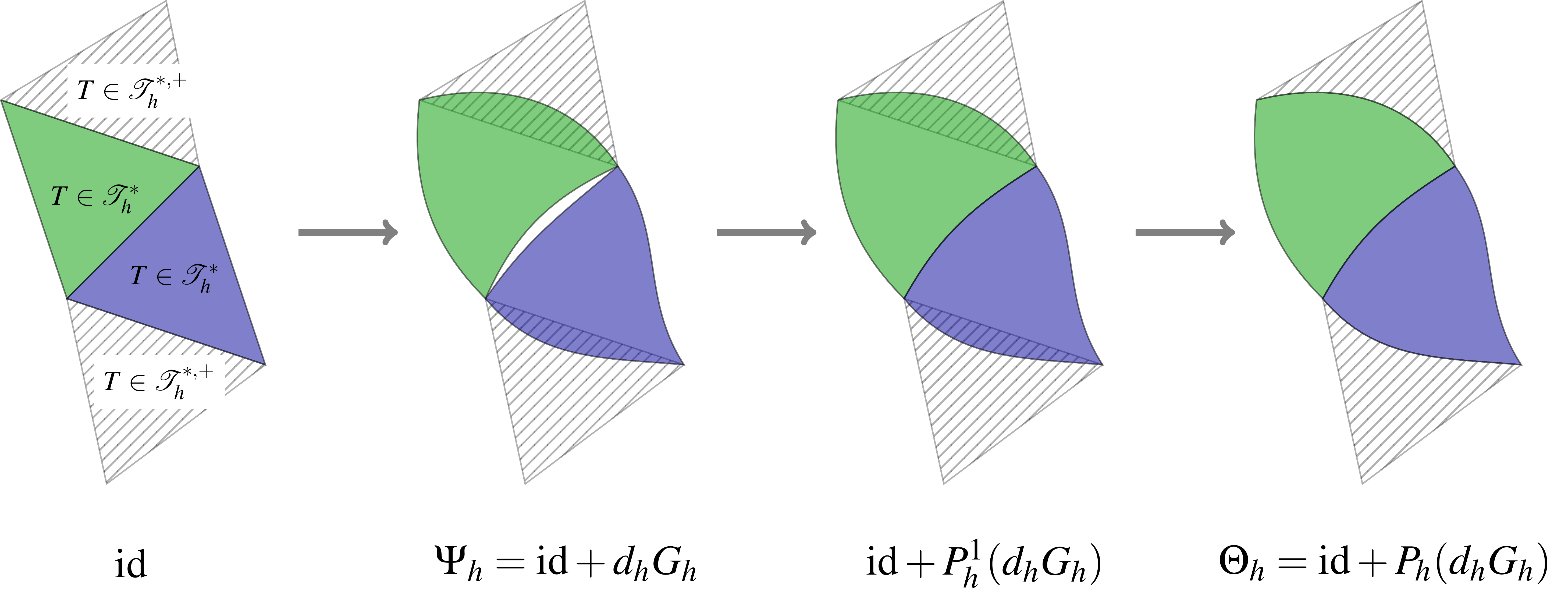} 
  \end{center}
  \vspace*{-0.4cm}
  \caption{Construction steps of the transformation $\Theta_h$. In the first step, $\Psi_h$ (only in $\mT_h^\Gamma$, pointwise, discontinuous across element interfaces) is constructed. In a second step, the discontinuities are removed through averaging (only in $\mT_h^\Gamma$). Finally, a continuous extension to the exterior is realized.}
  \vspace*{-0.4cm}
  \label{fig:proj} 
\end{figure}
A sketch of the steps in the construction of $\Theta_h$ is given in Fig. \ref{fig:proj}.
In the next two paragraphs we define the projections $P_h^1$ and $P_h^2$. For further details we refer to \cite{LR16a}.

\subsubsection*{$P_h^1$: An Oswald-type projection}
Consider $ v \in  C(\mT_h^\Gamma)$. We explain how we determine $P_h^1 v$.
Let $\{\varphi_i\}_{i=1,..,N}$ be the basis of the finite element space $V_h^k|_{\Omega^\Gamma}$,
$$
S_i := \{ T \in \mT_h^\Gamma \mid \operatorname{supp}(\varphi_i)\cap T \neq \emptyset\}
$$
be the set of elements where $\varphi_i$ is supported with $\# S_i$ the cardinality of $S_i$ and
$$
S_T := \{ i \in \{1,..,N\} \mid \operatorname{supp}(\varphi_i)\cap T \neq \emptyset\}
$$
the set of unknowns which are supported on $T$.
With a local interpolation we determine a polynomial approximation of $v$ on $T \in \mT_h^\Gamma$ which we can write as
$$
\sum_{i \in S_T} c_{i,T}(v) \varphi_i|_T
$$
with unique coefficients $c_{i,T}(v)$.
The obtained piecewise polynomial approximation which can be discontinuous across different elements in $\mT_h^\Gamma$ can be reformulated as
$$
\sum_{T\in \mT_h^\Gamma} \sum_{i \in S_T} c_{i,T}(v) \varphi_i|_T = \sum_{i = 1}^N \sum_{T\in S_i}  c_{i,T}(v) \varphi_i|_T.
$$
To obtain a continuous function on $\Omega^\Gamma$, we apply a simple averaging to define $P_h^1 v$:
$$
P_h^1 v :=  \sum_{i = 1}^N \frac{\sum_{T\in S_i}  c_{i,T}(v)}{\# S_i} \varphi_i|_{\Omega^\Gamma}.
$$
This type of projection is often called \emph{Oswald} interpolation, see also \cite{ernguermond15,oswald}.

\subsubsection*{$P_h^2$: A finite element extension procedure based on hierarchical basis functions}
We aim for a projection operation so that $P_h^2 (P_h^1\Psi_h)=\operatorname{id}$ on every element $T \in \mT_h \setminus \mT_h^{\Gamma,+}$, i.e. that elements away from the interface stay uncurved. Additionally we need the extension to be smooth so that also higher order derivates stay (uniformly) bounded.
Let $D_h := P_h^1\Psi_h - \operatorname{id} = P_h^1 (d_h G_h)$ be the mesh \emph{deformation} in $(V_h^k|_{\Omega^\Gamma})^d$.
We seek for the extension $\mathcal{E}_h: V_h^k|_{\Omega^\Gamma} \to V_h^k$ which is zero in $\mT_h \setminus \mT_h^{\Gamma,+}$ and leaves the values in $\mT_h^\Gamma$ unchanged. We then apply the extension (component-wise) to define $P_h^2(P_h^1\Psi_h) := \operatorname{id} + \mathcal{E}_h D_h$.

Next, we discuss the construction of this extension procedure which is based on a hierarchical decomposition of the basis functions of the finite element space $V_h^k$. If a more general basis for $V_h^k$ is considered, for instance a Lagrange basis, an equivalent definition of the extension can be given using blending techniques introduced in \cite{bernardi1989optimal,lenoir1986optimal}. We refer the interested reader to \cite[Section 3.3]{LR16a}.

Let $V_h^k = \bigoplus_{l=1}^k \bar{V}_h^l$ be the decomposition into the subspaces $\bar{V}_h^l$ so that $V_h^n = \bigoplus_{l=1}^n \bar{V}_h^l$ for all $n=1,..,k$.
We assume hierarchical basis functions, i.e. the basis of $V_h^k$ is obtain by adding the bases of $\bar{V}_h^l$ for $l=1,..,k$. This decomposition implies in particular that the usual piecewise linear hat functions of $\bar{V}_h^1 = V_h^1$ are also basis functions in $V_h^k$, $k>1$. Note that this is not the case for a Lagrange basis.

We are given a function $d_h \in V_h|_{\Omega^\Gamma}$ (e.g. one component of $D_h$) and can express $d_h$ in the (hierarchical) basis of $V_h^k$ as
$$
d_h = \sum_{i \in S^\Gamma} d_i \varphi_i|_{\Omega^\Gamma},~d_i\in\rr,\text{ where } S^\Gamma \! := \{ i \in \{1,..,N\} | \operatorname{supp}(\varphi_i) \cap \Omega^\Gamma \neq \emptyset \}
$$
is the set of unknowns whose basis functions are supported in $\Omega^\Gamma$.
We define the extension
$$
\mathcal{E}_h d_h := \sum_{i \in S^\Gamma} d_i \varphi_i 
$$
which coincides with setting $d_j = 0$ for $j \not\in S^\Gamma$. Obviously, the implementation of this extension is trivial, once $d_h$ is given in terms of the hierarchical basis functions.

We note that the hierarchical structure of the basis functions are crucial for this extension to provide the necessary (uniform) bound on the derivatives, see \cite[Section 3.3]{LR16a} for details in the analysis. 
To illustrate this, consider the example of a triangle $T$ in $\mT_h^{\Gamma,+} \setminus \mT_h^{\Gamma}$ with one edge $F$ adjacent to $\Omega^\Gamma$. If $d_h$ is a polynomial of degree $l$ on $F$, $d_h \in \mathcal{P}^l(F)$ for $l<k$, we will have that $\mathcal{E}_h d_h \in \mathcal{P}^l(T)$ (as coefficients to higher order basis functions will be zero for $d_h$ and thus also for $\mathcal{E}_h d_h$), i.e. the extension preserves the polynomial degree. This is not true if a Lagrange basis is used.
In the case of a Lagrange basis we will typically have $\Vert D^{l+1} \mathcal{E}_h d_h \Vert_{L^{\infty}(T)} > 0$ even if $\Vert D^{l+1} d_h \Vert_{L^{\infty}(F)} = 0$ so that we can not control the higher derivatives of $\mathcal{E}_h d_h$ by corresponding derivatives in $d_h|_{\Omega^\Gamma}$.

\subsection{Properties of the mapping}
In \cite{lehrenfeld15,LR16a,LR16b,GLR16} the mapping $\Theta_h$ and the resulting geometry approximation $\Omega_h$ has been analyzed. In this section we summarize the most important properties.

\begin{lemma} \label{propertiesdh}
  For $h$ sufficiently small, with $\Theta_h$ as in \eqref{eq:thetah} and $\Psi$ as in \eqref{eq:psi}, there holds
  \begin{subequations}
    \begin{align}
      \Theta_h(x)  =x \quad \text{for $x=x_V$ {\rm vertex} in $\mT_h$ or $~x \in \Omega^\mT \setminus \Omega^{\Gamma,+}$}, \label{resd3} \\
      \| \Theta_h - \operatorname{id} \|_{\infty}  \lesssim h^2, \qquad \qquad
      \| D \Theta_h - \operatorname{I}\|_{\infty} \lesssim h,  \qquad\qquad \label{resd4} \\
  \Vert \thetah - \Psi \Vert_{\infty,\Omega^{\Gamma}} + h \Vert D (\thetah - \Psi) \Vert_{\infty,\Omega^{\Gamma}} 
      \lesssim h^{k+1}
      \Rightarrow
\operatorname{dist}(\Gamma,\Gamma_h) \lesssim h^{k+1}.
    \end{align}
  \end{subequations}
\end{lemma}
\begin{proof} See Lemmas 3.4, 3.6 and 3.7 in \cite{LR16a}. \end{proof}
Note that $\Theta_h$ can be seen as a small perturbation to the identity. In the most part of the domain the deviation is zero and in the vicinity of the domain boundary it is small. As the transformation only ``repairs'' approximation errors of $\Gammalin$, the deviation from the identity decreases for $h \rightarrow 0$ as the approximation quality of $\Gammalin$ increases. Further, the geometry approximation after mesh transformation is of higher order accuracy, cf. the sketch in Fig. \ref{fig:trafos}.

 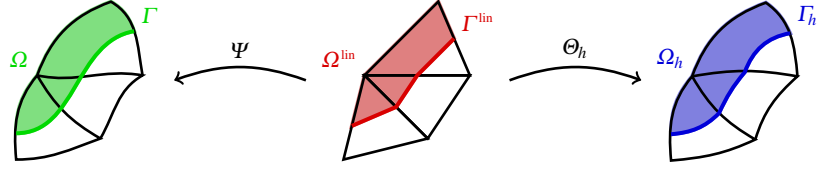
\begin{figure}[h!]
   \vspace*{-0.4cm}
   \begin{center}
     \begin{tikzpicture}[scale=1.4]
       \node(Olin) {
     \begin{tikzpicture}[scale=1.4]
       \coordinate (A) at (0.0,0.0);
       \coordinate (B) at (0.7,0.7);
       \coordinate (C) at (1.0,0.0);
       \coordinate (D) at (0.6,-0.6);
       \coordinate (E) at (-0.2,-0.8);
       \coordinate (F) at ($(A)!0.6!(E)$);
       \coordinate (G) at ($(A)!0.5!(D)$);
       \coordinate (H) at ($(A)!0.5!(C)$);
       \coordinate (I) at ($(B)!0.5!(C)$);
       \filldraw[red!75!black, opacity=0.5,line width=1.5pt] (F) -- (G) -- (H) -- (I) -- (B) -- (A) -- cycle;
       \draw[line width=1pt] (A) -- (B) -- (C) --cycle;
       \draw[line width=1pt] (A) -- (C) -- (D) --cycle;
       \draw[line width=1pt] (A) -- (D) -- (E) --cycle;
       \draw[red!85!black,line width=1.5pt] (F) -- (G) -- (H) -- (I);
       \node[above right] at (I) {\color{red!85!black}$\Gammalin$};
       \node[above left] at (A) {\color{red!85!black}$\Omegalin$};
     \end{tikzpicture}
   };
   \node[left= 1.75cm of Olin.west](Oexa) {
     \begin{tikzpicture}[scale=1.4]
       \coordinate (A) at (0.0,0.0);
       \coordinate (B) at (0.7,0.7);
       \coordinate (C) at (1.0,0.0);
       \coordinate (D) at (0.6,-0.6);
       \coordinate (E) at (-0.2,-0.8);
       \coordinate (F) at (-0.2,-0.55);
       \coordinate (I) at (0.92,0.42);
       \filldraw[green!75!black, opacity=0.5,line width=1.5pt] (F) to[in=-170,out=0] (I) to[in=-30,out=110] (B) to[out=-160,in=70] (A) to[in=90,out=-130] (F);
       \draw[line width=1pt] (A) to[in=-160,out=70] (B) to[in=100,out=-30] (C);
       \draw[line width=1pt] (A) to[in=170,out=-12] (C) to[in=60,out=-150] (D);
       \draw[line width=1pt] (A) to[in=150,out=-60] (D) to[in=0,out=-160] (E) to[in=-130,out=95] (A);
       \draw[green!85!black,line width=1.5pt] (F) to[in=-170,out=0] (I);
       \node[above right] at (I) {\color{green!85!black}$\Gamma$};
       \node[above left] at (A) {\color{green!85!black}$\Omega$};
     \end{tikzpicture}
   };
   \node[right= 1.75cm of Olin.east](Oh) {
     \begin{tikzpicture}[scale=1.4]
       \coordinate (A) at (0.0,0.0);
       \coordinate (B) at (0.7,0.7);
       \coordinate (C) at (1.0,0.0);
       \coordinate (D) at (0.6,-0.6);
       \coordinate (E) at (-0.2,-0.8);
       \coordinate (F) at (-0.2,-0.55);
       \coordinate (G) at (0.255,-0.35);
       \coordinate (H) at (0.5,0.04);
       \coordinate (I) at (0.93,0.41);
       \filldraw[blue!75!black, opacity=0.5,line width=1.5pt] (F) to[in=-135,out=5] (G) to[in=-130,out=65] (H) to[out=70,in=-170] (I) to[in=-35,out=115] (B) to[out=-160,in=70] (A) to[in=90,out=-130] (F);
       \draw[line width=1pt] (A) to[in=-160,out=70] (B) to[in=95,out=-35] (C);
       \draw[line width=1pt] (A) to[in=170,out=10] (C) to[in=73,out=-138] (D);
       \draw[line width=1pt] (A) to[in=150,out=-60] (D) to[in=5,out=-155] (E) to[in=-130,out=95] (A);
       \draw[blue!85!black,line width=1.5pt] (F) to[in=-135,out=5] (G) to[in=-130,out=65] (H) to[out=70,in=-170] (I);
       \node[above right] at (I) {\color{blue!85!black}$\Gamma_h$};
       \node[above left] at (A) {\color{blue!85!black}$\Omega_h$};
     \end{tikzpicture}
   };
   \draw[line width=1pt,->] (Olin.east) to[in=160,out=20] node[above] {$\Theta_h$}  (Oh.west);
   \draw[line width=1pt,->] (Olin.west) to[in=20,out=160] node[above] {$\Psi$} (Oexa.east);

   \end{tikzpicture}
   \end{center}
   \vspace*{-0.4cm}
   \caption{Sketch of the transformations $\Psi$ and $\Theta_h$ on $\mT_h^{\Gamma}$. $\Psi$ maps the piecewise linear domains on the exact geometries, $\thetah$ is the discrete transformation approximating $\Psi$.
   }
   \vspace*{-0.5cm}
   \label{fig:trafos}
 \end{figure}

\begin{remark}[Shape regularity]
Due to the properties of $\Theta_h$ in Lemma \ref{propertiesdh}, it is easy to deduce shape regularity of the deformed mesh based on shape regularity of the original background mesh for $h$ sufficiently small. It can however happen that shape regularity is not obtained on very coarse grids where $\Gammalin$ is not able to resolve the geometry sufficiently well. In this case a modification of the mesh transformation has to be applied, cf. \cite{lehrenfeld15} for a possible remedy. In the remainder we assume that the shape regularity of the deformed mesh is not an issue and that $h$ is sufficiently small.
\end{remark}

We also have bounds for the higher derivatives of the mapping (and its inverse).
\begin{lemma} \label{lem:invtrafo}
  For $T \in \mathcal{T}_h$ we have 
  $$
\Vert \Theta_h \Vert_{m,\infty,T} \lesssim 1 \quad \text{ and } \quad
\Vert \Theta_h^{-1} \Vert_{m,\infty,\Theta(T)} \lesssim 1, \quad m \in \{1,..,k+1\},
$$
with $\Vert  \cdot \Vert_{m,\infty,\Theta(T)} := \max_{l\leq m} \Vert D^l \cdot \Vert_{L^{\infty}(T)}$.
\end{lemma}
\begin{proof}The proof is given in the appendix, section \ref{sec:proofs}.
\end{proof}

In the finite element formulation in section \ref{sec:isoparammethod} we will lift functions from $V_h^k$ that are defined on $\Omegalin$ to functions defined on $\Omega_h$. It is useful to know that $L^2$ and $H^1$ norms on domains from the reference configuration $\Omegalin$ and the mapped configuration $\Omega_h$ are equivalent. For this, we have the following result. \\
\begin{lemma}\label{lemF}
  For $v,w \in H^1(\Omega_h)$ with $\operatorname{tr}|_{\Gamma_h} \nabla w \in L^2(\Gamma_h)$
  there holds
\begin{subequations}
\begin{align}
  \Vert v \circ \thetah \Vert_{\Omegalin}^2 & \simeq \Vert v \Vert_{\Omega_h}^2 & \quad \text{and} \quad  && \Vert v \circ \thetah \Vert_{\Gammalin}^2 & \simeq \Vert v \Vert_{\Gamma_h}^2, \\
  \Vert \nabla (w \circ \thetah) \Vert_{\Omegalin}^2 & \simeq \Vert \nabla w \Vert_{\Omega_h}^2 & \quad \text{and} \quad &&  \Vert \nabla (w \circ \thetah) \Vert_{\Gammalin}^2 & \simeq \Vert \nabla w \Vert_{\Gamma_h}^2.     \end{align}
\end{subequations}
\end{lemma}
\begin{proof} Follows from Lemma 3.3 \ in \cite{LR16a}. \end{proof}

Below, in the isoparametric finite element discretization we work with the domain approximation $\Omega_h = \Theta_h(\Omegalin)$. The original partial differential equation is however formulated with respect to the domain $\Omega$. For the analysis later on we need a (smooth) bijection $\Phi_h$ on $\widetilde\Omega$ which has the property $\Phi_h(\Gamma_h) = \Gamma$. $\Phi_h := \Psi \circ \Theta_h^{-1}$ has the property $\Phi_h( \Gamma_h) = \Gamma$ and is piecewise smooth. Further, it is a higher order perturbation to the identity.
\begin{lemma} \label{lem:phih} For $h$ sufficiently small $\Phi_h = \Psi \circ \Theta_h^{-1}:\, \widetilde\Omega \to \widetilde\Omega$ defines a homeomorphism with $\Phi_h \in C(\widetilde\Omega) \cap C^{k+1}(\Theta_h(\widetilde{\mathcal{T}}_h))$, which has the following properties:
\begin{subequations}
\begin{align}
\Phi_h(\Omega_h) =\Omega, \\
\Vert \id - \Phi_h \Vert_{\infty,\widetilde\Omega} + h \Vert I - D \Phi_h \Vert_{\infty,\widetilde\Omega} & \lesssim h^{k+1}. \label{eq:estphih}
\end{align}
\end{subequations}
\end{lemma}
\begin{proof}
  See Lemma 5.3 in \cite{LR16a}.
\end{proof}

\section{Isoparametric fictitious domain finite element formulation} \label{sec:isoparammethod}
With respect to the geometry approximations $\Omega_h$ and $\Gamma_h$  we  want to formulate a finite element method based on Nitsche's method. We consider the classic version of Nitsche's method which is symmetric. Variations as in \cite{boiveau2016fictitious} are also possible. 
Let $V_h$ be the space $V_h^k$ restricted to $\Omega^\mT$, $V_h:=\{v \in H^1(\Omega^\mT) \mid  v|_{T} \in P^k(T), T \in \mT^\Gamma \}$, where $P^k(T)$ is the space of polynomials on $T$ up to degree $k$.
Induced by the parametric mapping $\Theta_h$ we define a corresponding isoparametric finite element space 
\begin{equation}
  \mathcal{V}_h := \{v_h \circ \Theta_h^{-1}, v_h \in V_h\} = \{ v_h, v_h \circ \Theta_h \in V_h\}.
\end{equation}
We further introduce the infinite dimensional space which allows for the evalution of the normal derivative on the discrete boundary $\Gamma_h$,
\begin{equation}
  \Vreg := H^2(\Omega_h) \oplus \{u \circ \Phi_h \mid u \in H^2(\Omega)\}.
\end{equation}
Next, we apply the Nitsche variational formulation to discretize the problem based on $\Omega_h$, $\Gamma_h$ and $\mathcal{V}_h$.
We define the isoparametric fictitious domain finite element method as: \\
Find $u \in \mathcal{V}_h$, such that for all $v \in \mV_h$ there holds
\begin{equation}\label{eq:Nitsche1}
  B_h(u,v) := A_h(u,v) + J_h(u,v) := a_h(u,v) + N_h(u,v) + J_h(u,v) = f_h(v),
\end{equation}
with the following bi- and linear forms.
For the Nitsche formulation we introduce the following bilinear forms which are well-defined for $u,v \in \mV_h + \Vreg$
\begin{subequations} \label{eq:blflfs}
  \begin{align}
    a_h(u,v) & := \int_{\Omega_h} \nabla u \cdot \nabla v \, dx, & \hspace*{-6cm} & \\ 
   N_h(u,v) & := N_h^c(u,v) + N_h^c(v,u) + N_h^s(u,v) \text{ with } \hspace*{-6cm}\\
  N_h^c(u,v) & := \int_{\Gamma_h} (- \partial_n u) v \ ds, &  \hspace*{-6cm} 
N_h^s(u,v) & := \frac{\lambda}{h} \int_{\Gamma_h} u v \ ds. \\
    \intertext{
    Here, $\lambda$ is a parameter of the Nitsche method. 
We note that $N_h^c(v,w)$ and $N_h^s(w,v)$ are also well defined for $v \in \mV_h + \Vreg$ and $w \in L^2(\Gamma)$ and further define
}
             f_h(v) & :=  \int_{\Omega_h} f^e v \, dx + N_h^c(v,u_D^e) + N_h^s(u_D^e,v), \quad v \in \mV_h + \Vreg. \hspace*{-4cm} 
  \end{align}
  Here, we used $f^e$ and $u_D^e$, extension of $f$ and $u_D$ which we briefly discuss.
  As $f$ is only defined on $\Omega$ and the discretization is defined on $\Omega_h$, we assume that the source term $f$ is (smoothly) extended to $\Omega^e = \Omega_h \cup \Omega$ such that $f^e= f$ on $\Omega$ holds. Correspondingly we assume that $u^e_D$ is a (smooth) extension of the boundary data from $\Gamma$ to a neighborhood $\Gamma^e \supset \Gamma \cup \Gamma_h$ so that $u^e_D = u_D$ on $\Gamma$.

  We split the bilinear form $A_h=A_h^1+A_h^2$ and the linear form $f_h=f_h^1+f_h^2$ with 
  \begin{align}
    A_h^1(u,v) & := a_h(u,v) + N_h^c(u,v), & A_h^2(u,v) & = N_h^c(v,u) + N_h^s(u,v), \label{eq:split1}\\
      f_h^1(v) & := \int_{\Omega_h} f^e v \, dx \quad \text{ and } & f_h^2(v) & = N_h^c(v,u_D^e) + N_h^s(u_D^e,v). \label{eq:split2}
  \end{align}
The bilinear form $A_h^1$ together with $f_h^1$ is responsible for consistency whereas $A_h^2$ and $f_h^2$ are added for (consistently) realizing symmetry and additional control of the boundary values.
To provide stability in the case of small cuts we add the (higher order) ghost penalty stabilization with the bilinear form
\begin{equation}
  \hspace*{-0.125cm}  J_h(u,v) := \sum_{l=1}^k \gamma_l \!\!\!\!\sum_{F \in \mF_{h}}\!\!\!\! h^{2l-1} \!\!\!\!\! \int_{\Theta_h(F)} \spacejump{ \partial_n^l u}  \spacejump{ \partial_n^l v}  \ ds, ~~~u,v \in \mV_h.
\end{equation}
\end{subequations}
Here, $\gamma_l$ are stabilization parameters indepedent of $h$, $\spacejump{\cdot}$ denotes the usual jump operator across (curved) element interfaces and $\partial_n^l$ is the $l$-th directional derivative in the direction normal to a facet $\Theta_h(F)$.
The stabilization introduces properly scaled penalties for jumps in (higher order) weak discontinuities across element interfaces close to the domain boundary. 
Thereby stiffness between the values in $\Theta_h(\Omega^\mT)$ and $\Omega_h$ is introduced independently of the cut position, cf. Lemma \ref{lem:gp} below. 
The benefit of this stabilization is twofold. On the one hand, trace inverse inequalities that are required in (the analysis of) Nitsche's method can be applied independent of the shape regularity of the cut elements. On the other hand, the conditioning of arising linear systems is robust with respect to the position of the cuts within the elements. A significant drawback of this stabilization is the fact that the coupling relations of degrees of freedom change which changes the sparsity structure of arising linear systems. As the facet-based stabilization is only added in the vicinity of the domain boundary this drawback is often outweighed by the benefits of the stabilization.

We note that in an implementation of the method, all integrals in \eqref{eq:blflfs} will be computed after transformation to the reference domains $\Omegalin$, $\Gammalin$ and $F$ where (higher order) quadrature rules for straight cuts can be applied.

\section{A priori error estimate} \label{sec:analysis}

According to the bi- and linear forms of the discrete variational formulation we introduce proper norms that we will use in the analysis below. 
\begin{subequations}
\begin{align} \label{eq:norms}
  \Vert v \Vert_A^2 := &a_h(v,v) + \Vert v \Vert_{\frac12,h,\Gamma_h}^2 + \Vert \partial_n v \Vert_{-\frac12,h,\Gamma_h}^2, \quad & v & \in \mV_h + \Vreg, \\
\text{ with } & \quad \Vert v \Vert_{\pm\frac12,h,\Gamma_h}^2  := h^{\mp 1} \Vert v \Vert_{\Gamma_h}^2, \quad & v & \in L^2(\Gamma_h), \nonumber \\
 \Vert v \Vert_J^2 := &J_h(v,v), \quad \Vert v \Vert_B^2 := \Vert v \Vert_A^2 + \Vert v \Vert_J^2, \quad & v & \in \mV_h.
\end{align}
\end{subequations}
The $\Vert \cdot \Vert_A$-norm is the norm that is usually used to analyse Nitsche-type methods whereas $\Vert \cdot \Vert_J$ corresponds to the discrete energy induced by the ghost penalty terms.
We note that norms are defined with respect to the domain $\Omega_h$ and the (curved) facets $\Theta_h(F)$, $F \in \mathcal{F}_h$. The analysis combines a Strang-type strategy for the estimation of geometry errors with techniques from the analysis of ghost penalty type discretizations, cf. for instance \cite{burman2012fictitious,MassingLarsonLoggEtAl2013a}.
We start with gathering versions of Galerkin orthogonality, coercivity and continuity in sections \ref{sec:galorth}, \ref{sec:coerc} and \ref{sec:conti} to apply a Strang-type lemma in section \ref{sectStrang}. It then remains to bound consistency errors in section \ref{sec:integrals}. Finally, we show best approximation of the solution in $\mathcal{V}_h$ to obtain (almost) optimal error bounds in the $\Vert \cdot \Vert_A$-norm in section \ref{sec:approx}.

\subsection{Galerkin orthogonality} \label{sec:galorth}
For later reference, when analysing errors stemming from the geometry approximation, we introduce bi- and linear forms with respect to the exact geometry.
In order to define discrete functions on the exact geometry we make use of the mapping $\Phi_h$ from Lemma \ref{lem:phih} which maps $\Omega_h$ on $\Omega$.
For $ u, v \in H^2(\Omega) \oplus \{v \circ \Phi_h^{-1}, v \in \mV_h\}$ we define 
\begin{equation}
  A(u,v) := \int_{\Omega} \nabla u \nabla v \, dx + \int_{\partial \Omega} (-\partial_n u) v \, ds,
  \quad \text{ and } \quad
  f(v) := \int_{\Omega} f^e v \, dx.
\end{equation}
\begin{lemma} \label{lem:galorth}
  Let $u \in H^2(\Omega)$ be the solution to \eqref{eq:ellmodel}. Then there holds
  \begin{equation}
    A(u,w_h \circ \Phi_h^{-1}) = f(w_h \circ \Phi_h^{-1}) \quad \text{for all } w_h \in \mathcal{V}_h.
  \end{equation}
\end{lemma}
\begin{proof}
  Applying partial integration on the  volume integral in $A(\cdot,\cdot)$ we obtain
  $\int_{\Omega} (-\Delta u - f^e) v\, dx = \int_{\Omega} (-\Delta u - f) v\, dx = 0$ and the boundary integral stemming from partial integration cancels out with the remaining boundary integral in $A(\cdot,\cdot)$.
\end{proof}

\subsection{Coercivity}\label{sec:coerc}
Due to the mesh transformation the functions in $\mathcal{V}_h$ are no longer piecewise polynomials of degree $k$, so that in general there holds $D^{k+1}v \neq 0$ for $v \in \mathcal{V}_{h}$. Nevertheless, we have that $D^{k+1}v$ is small in $\mathcal{T}_h^{\Gamma,+}$ in the following sense.
\begin{lemma} \label{lemhighder}
  For $v \in \mathcal{V}_{h}$ and $T \in \mathcal{T}_h^{\Gamma,+}$ there holds
  \begin{equation}
    \Vert D^{k+1} v \Vert_{\infty,\Theta(T)} \lesssim h^{-k} \Vert v \Vert_{\infty,\Theta(T)}.
  \end{equation}
\end{lemma}
\begin{proof}The elementary proof is given in the appendix, section \ref{sec:proofs}.
\end{proof}
Before we can prove coercivity, we need an adapted version of Theorem 5.1 in \cite{MassingLarsonLoggEtAl2013a} which characterizes the ghost penalty mechanism.
\begin{lemma}\label{lem:gp}
  For two neighboring (curved) elements $T_i=\Theta_h(\hat{T}_i),~i=1,2$, $\hat{T}_i \in \mathcal{T}_h^{\Gamma,+},~i=1,2$ and $F$ the dividing (curved) element interface $F = \Theta(\hat{F})$, $\hat{F} \in \mathcal{F}_h$ and
  $v \in \mathcal{V}_h$ there holds
\begin{subequations}
\begin{align}
    \Vert v \Vert_{T_1}^2 & \lesssim \Vert v \Vert_{T_2}^2 + \sum_{l=1}^k h^{2l+1} \int_F \spacejump{\partial_n^l v} \spacejump{\partial_n^l v} \, ds , \label{eq:a} \\
    \Vert \nabla v \Vert_{T_1}^2 & \lesssim \Vert \nabla v \Vert_{T_2}^2 + \sum_{l=1}^k h^{2l-1} \int_{F} \spacejump{\partial_n^l v} \spacejump{\partial_n^l v}\, ds, \label{eq:b}
\end{align}
\end{subequations}
for sufficiently small mesh sizes $h$. 
\end{lemma}
\begin{proof} The proof which is based on a Taylor series expansion and the result of Lemma
  \ref{lemhighder} is provided in the appendix, section \ref{sec:proofs}.
\end{proof}

With these preparations we can show coercivity.
\begin{lemma} \label{lem:coercivity}
  For $\lambda$ sufficiently large and $h$ sufficiently small, so that Lemma \ref{lem:gp} holds true, there holds
  \begin{equation}
    B_h(u_h,u_h) \gtrsim \Vert u_h \Vert_A^2 + \Vert u_h \Vert_J^2 \quad \text{for all } u_h \in \mathcal{V}_h.
  \end{equation}
\end{lemma}
\begin{proof}
  Let $\Gamma_{h,T} := \Gamma_h \cap \Theta_h(T)$, $T\in \mT_h^\Gamma$ and $\hat{u}_h := u_h \circ \Theta_h$. Cauchy-Schwarz and Young's inequality yield
  \begin{align*}
    2 \int_{\Gamma_{h,T}} u_h \partial_n u_h \, ds & \leq \frac{h}{\gamma} \Vert \partial_n u_h
\Vert_{\Gamma_{h,T}}^2 + \frac{\gamma}{h} \Vert u_h \Vert_{\Gamma_{h,T}}^2,
  \end{align*}
  for a $\gamma > 0$.
  For the former part we transform the integral to $\Gammalin_T := \Gammalin \cap T$, exploiting Lemma \ref{lemF}, where we apply an inverse trace inequality (with respect to a planar cut configuration) and transform the integral back to $\Theta_h(T)$:
  \[
  h  \Vert \partial_n u_h \Vert_{\Gamma_{h,T}}^2 \lesssim
  h  \Vert \nabla \hat{u}_h  \Vert_{\Gammalin_{T}}^2 \lesssim \Vert \nabla \hat{u}_h \Vert_{T}^2 \lesssim
  \Vert \nabla u_h \Vert_{\Theta_h(T)}^2.
  \]
  Hence, there is a constant $c_{tr}>0$ only depending on the shape regularity of the mesh, so that 
  \begin{equation}\label{ctr}
    2 \int_{\Gamma_{h,T}} u_h \partial_n u_h \, ds \leq \frac{c_{tr}}{\gamma}  \Vert \nabla u_h \Vert_{\Theta_h(T)}^2 + \frac{\gamma}{h} \Vert u_h \Vert_{\Gamma_{h,T}}^2.
  \end{equation}
  We note that we bounded the boundary integral term using the full element $\Theta_h(T)$ and not only the part in $\Omega_h$.
  Here, we need the ghost penalty term to relate this to the parts in $\Omega_h$. For this, we use the fundamental result for the ghost penalty method which
  has been proven for instance in \cite{MassingLarsonLoggEtAl2013a} for uncurved meshes. With Lemma \ref{lem:gp} the result holds true also in the case of curved elements, i.e. there holds
  \begin{equation}\label{gpa}
    \Vert \nabla v_h^2 \Vert_{\Omega_h} + J_h(v_h,v_h) \gtrsim \Vert \nabla v_h \Vert_{\Theta_h(\Omega^\mT)}^2, \quad \forall v_h \in \mathcal{V}_h.
  \end{equation}
  With \eqref{ctr} and \eqref{gpa}, there exists $\gamma>0$ (independent of $h$ and the cut position) such that there holds
  \begin{equation*}
   2 N_h^c(u_h,u_h) \leq \frac12 a(u_h,u_h) + \frac12 J(u_h,u_h) + \frac{\gamma}{h} \Vert u_h \Vert_{\Gamma_h}^2.
  \end{equation*}
  With $\lambda \geq 2 \gamma$ we have
  \begin{align*}
    B_h(u_h,u_h) & \geq \frac12 \left( a(u_h,u_h) + J(u_h,u_h) + N_h^s(u_h,u_h) \right).
  \end{align*}
  Finally the claim follows from the trace inverse inequality, Lemma \ref{lemF} and \eqref{gpa}:
  \begin{align*}
   \Vert \partial_n u_h \Vert_{-\frac12,h,\Gamma_h}^2 \lesssim  h \Vert \partial_n \hat{u}_h \Vert_{\Gammalin}^2 \lesssim \Vert \nabla \hat{u}_h \Vert_{\Omega^\mT}^2 \lesssim \Vert \nabla u_h \Vert_{\Theta_h(\Omega^\mT)}^2 \lesssim a(u_h,u_h) + J_h(u_h,u_h).
  \end{align*}
\end{proof}
As a direct consequence of Lemma \ref{lem:coercivity} we know that \eqref{eq:Nitsche1} has a unique solution in $\mV_h$.
We note that $\lambda$ has to be chosen ``sufficiently large'', $\lambda > \lambda_0$, where $\lambda_0$ depends on the shape regularity of the background mesh $\widetilde{\mathcal{T}}_h$ and the ghost penalty stabilization parameters $\gamma_l$.
\subsection{Continuity}\label{sec:conti}
\begin{lemma} \label{lem:cont}
  There holds
  \begin{subequations}
  \begin{align}
    A_h(u,v) \lesssim \Vert u \Vert_A \Vert v \Vert_A \quad & \text{for all } u,v \in \mathcal{V}_h + \Vreg, \\
    J_h(u,v) \lesssim \Vert u \Vert_J \Vert v \Vert_J \quad & \text{for all } u,v \in \mathcal{V}_h.
  \end{align}
  \end{subequations}
\end{lemma}
\begin{proof} Follows from the definition of the norms and Cauchy-Schwarz inequalities.
\end{proof}
We note that we will only require continuity of the ghost penalty bilinear form on the discrete (finite element) space.
\vspace*{-0.5cm}
\subsection{Strang lemma} \label{sectStrang}

\begin{lemma}\label{lem:strang}
Let $u\in H^1(\Omega)$ be the solution of \eqref{eq:ellmodel} and $u_h \in \mathcal{V}_{h}$ the solution of \eqref{eq:Nitsche1}. The following holds: \vspace*{-0.1cm}
\begin{subequations}\label{eq:strang}
\begin{align}
\Vert u \circ \Phi_h - u_h \Vert_A \lesssim &  \inf_{v_h \in \mathcal{V}_h} \left(  \Vert u\circ \Phi_h - v_h \Vert_A + \Vert  v_h \Vert_J  \right) \label{eq:str1}\\ 
 & + \sup_{w_h \in \mathcal{V}_h} \frac{ | f_h^1(w_h) - f(w_h \circ \Phi_h^{-1}) | }{\Vert w_h \Vert_A} \label{eq:str2}\\
 & + \sup_{w_h \in \mathcal{V}_h} \frac{ | A_h^1(u\circ \Phi_h,w_h) - A(u,w_h \circ \Phi_h^{-1}) | }{\Vert w_h \Vert_A} \label{eq:str3}\\
 & + \Vert u \circ \Phi_h - u_D^e \Vert_{\frac12,h,\Gamma_h}. \label{eq:str4} 
\end{align}
\end{subequations}
\end{lemma}
\begin{proof}
  The proof is similar to the proof of Lemma 5.12 in \cite{LR16a} except for the ghost penalty part and the treatment of $A_h^2$ and $f_h^2$. The concept is along the same lines as the well-known Strang Lemma. We use the notation $\tilde u =u \circ \Phi_h$ and start with the triangle inequality with an arbitrary $v_h \in \mathcal{V}_h$:
\begin{equation*}
\Vert \tilde u - u_h \Vert_A \leq \Vert \tilde u - v_h \Vert_A  + \Vert v_h - u_h \Vert_A.
\end{equation*}
With $\mathcal{V}_h$-coercivity, cf. Lemma \ref{lem:coercivity}, we have ($w_h:=u_h-v_h$)
\begin{equation*} 
\begin{split}
  \Vert w_h \Vert_A^2  + \Vert w_h \Vert_J^2 & \lesssim B_h(u_h - v_h,w_h) = 
f_h(w_h) - A_h(v_h,w_h) - J_h(v_h,w_h) \\
& \lesssim \left| A_h(\tilde u-v_h,w_h) \right| + \left| A_h(\tilde u,w_h) - f_h(w_h) \right| + \Vert v_h \Vert_J \Vert w_h \Vert_J.
\end{split}
\end{equation*}
Using continuity, cf. Lemma \ref{lem:cont}, and dividing by $\Vert w_h \Vert_B$ results in
\begin{equation*}
\Vert \tilde u - u_h \Vert_A \lesssim \inf_{v_h \in \mathcal{V}_h} \left( \Vert \tilde u - v_h \Vert_A + \Vert v_h \Vert_J \right) + \sup_{w_h \in \mathcal{V}_h} \frac{ | A_h(\tilde u,w_h) - f_h(w_h) | }{\Vert w_h \Vert_A}.
\end{equation*}
Using the consistency property of Lemma~\ref{lem:galorth} and the splitting of $A_h$ and $f_h$ in \eqref{eq:split1} and \eqref{eq:split2} yields \vspace*{-0.6cm}
\begin{align*}
& |A_h(\tilde u,w_h) - f_h( w_h)| 
 = |A_h(\tilde u,w_h) - f_h( w_h) - \overbrace{\left(A(u,w_h \circ \Phi_h^{-1}) - f(w_h\circ \Phi_h^{-1})\right)}^{=0}| \\ 
& \leq |A_h^1( \tilde u, w_h) - A(u, w_h \circ \Phi_h^{-1})| + |f_h^1( w_h) - f(w_h \circ \Phi_h^{-1})|  + |A_h^2( \tilde u, w_h) - f_h^2( w_h)|.
\end{align*}
Dividing by $\Vert w_h \Vert_A$ and using \vspace*{-0.1cm}
$$
|A_h^2( \tilde u, w_h) - f_h^2( w_h)| = \left| \int_{\Gamma_h} (-\partial_n w_h + \frac{\lambda}{h} w_h) ( \tilde u - u_D^e) \, ds \right| \lesssim \Vert w_h \Vert_A \Vert \tilde u - u_D^e \Vert_{\frac12,h,\Gamma_h}
$$
for the latter part completes the proof.
\end{proof}

\subsection{Consistency error bounds} \label{sec:integrals} 
We derive consistency error bounds for the right-hand side terms \eqref{eq:str2}-\eqref{eq:str4} in the Strang estimate.
\begin{lemma} \label{lemconsist}
  Let $u \in H^2(\Omega)$ be a solution of \eqref{eq:ellmodel}. We assume that $f \in H^{1,\infty}(\Omega)$, $u_D \in H^{1,\infty}(\Gamma)$ and the data extensions $f^e$ and $u_D^e$ (cf. section \ref{sec:isoparammethod}) satisfy $\|f^e\|_{1,\infty,\Omega^e} \lesssim \|f\|_{1,\infty,\Omega}$ 
and $\|u_D^e\|_{1,\infty,\Omega^e} \lesssim \|u_D\|_{1,\infty,\Gamma}$. Then, the following holds for $w_h \in \mathcal{V}_h$:
\begin{subequations}
\begin{align}
 |A(u,w_h\circ \Phi_h^{-1}) - A_h(u\circ \Phi_h,w_h)| & \lesssim h^k \Vert u \Vert_{H^2(\Omega)} \Vert w_h \Vert_A, \label{p1}
\\
  |f(w_h\circ \Phi_h^{-1})-f_h(w_h)| & \lesssim h^{k} \|f\|_{1,\infty,\Omega}  \| w_h \Vert_A, \label{p2} \\
  \Vert u \circ \Phi_h - u_D^e \Vert_{\frac12,h,\Gamma_h}  & \lesssim h^{k+\frac12} \|u_D\|_{1,\infty,\Gamma}.  \label{p3}
\end{align}
\end{subequations}
\end{lemma}
\begin{proof}
  The proofs of \eqref{p1} and \eqref{p2} follow the same lines as the proof of Lemma 5.13 in \cite{LR16a}.
  To obtain the bound \eqref{p3} we note that $u = u_D = u_D^e$ on $\Gamma$ so that with Lemma \ref{lem:phih} we get
  $$
  \Vert u \circ \Phi_h - u_D^e \Vert_{\frac12,h,\Gamma_h} \lesssim h^{-\frac12} \Vert \Phi_h - \id \Vert_{\infty,\Gamma_h} \Vert u_D^e \Vert_{1,\infty,\Omega^e}\lesssim h^{k+\frac12} \Vert u_D \Vert_{1,\infty,\Gamma}.
  $$
\end{proof}

\subsection{Approximation errors} \label{sec:approx}

We obtained reasonable bounds for the geometrical consistency errors. We now treat the approximability of solutions with $\mV_h$ in the $\Vert \cdot \Vert_B$-norm.
Due to the fact that the lift $u \circ \Phi_h$ is not (globally) smooth (higher derivatives will in general be discontinuous across curved facets) we introduce a (globally) smooth quantity  for the approximation with finite element functions. We note that this becomes necessary, in contrast to the analysis in \cite{LR16a}, because of the higher order jump terms in the ghost penalty.

\begin{lemma} \label{lemuext}
For $u \in H^{3,\infty}(\Omega)$ or $u \in H^{k+1}(\Omega)$, we define $u^e := E u$ where $E : H^{3,\infty}(\Omega) \rightarrow H^{3,\infty}(\widetilde\Omega)$ for $k=2$ and $E : H^{k+1}(\Omega) \rightarrow H^{k+1}(\widetilde\Omega)$ for $k\geq 3$ is a continuous extension operator as in \cite[Theorem II.3.3]{GaldibookNS}. Then there holds
\begin{equation} \label{io}
\Vert u\circ \Phi_h - u^e \Vert_A  \lesssim h^{k+\frac12} S(u),~\text{ with } S(u):= \begin{cases}
  \|u\|_{H^{3,\infty}(\Omega)} &~~\text{if}~~k=2, \\                                                                                
    \|u\|_{H^{k+1}(\Omega)} & ~~\text{if}~~k \geq 3. \end{cases}
\end{equation}
\end{lemma}
\begin{proof}
We first consider $\Vert \nabla(u^e - u \circ \Phi_h) \Vert_{L^2}$:
\begin{align*}
 & \Vert \nabla ( u^e - u \circ \Phi_h )\Vert_{L^2(\Omega_h)}^2 
 \lesssim \Vert \nabla (u^e \circ \Phi_h^{-1} - u^e )\Vert_{L^2(\Omega)}^2   \\
 & \lesssim | U_{\delta_h} |\big( \|u^e \Vert_{H^{1}(\Omega)}^2 \Vert D \Phi_h^{-1} - \id \Vert_{\infty,\Omega}^2 + \|u^e \Vert_{H^{2,\infty}(\Omega)}^2 \Vert \Phi_h^{-1} - \id \Vert_{\infty,\Omega}^2 \big) \lesssim h^{2k+1} S(u)^2,
\end{align*}
where $U_{\delta_h}$ is the domain where $\Phi_h \neq \id$, $| U_{\delta_h} | \lesssim h$.  Next, we consider the boundary term:
\begin{align*}
  \Vert u^e - u \circ \Phi_h \Vert_{L^2(\Gamma_h)}^2 & \lesssim \|u^e \circ \Phi_h^{-1} - u^e \|_{L^2(\Gamma)}^2 \\
 & \lesssim | \Gamma |  \Vert  u^e \Vert_{H^{1,\infty}(\Omega)}^2 \Vert \Phi_h^{-1} - \id \Vert_{\infty,\Omega}^2  \lesssim h^{2k+2} S(u)^2.
\end{align*}
Accordingly,
\begin{align*}
 & \Vert {\nabla ( u^e - u \circ \Phi_h) \cdot n } \Vert_{L^2(\Gamma_h)}^2 \\
 & \lesssim \Big(\Vert u^e \Vert_{H^{2}(\Omega)}^2 \Vert D\Phi_h^{-1} - \id \Vert_{\infty,\Omega}^2 + | \Gamma |\Vert u^e \Vert_{H^{2,\infty}(\Omega)}^2 \Vert \Phi_h^{-1} - \id \Vert_{\infty,\Omega}^2 \Big)\lesssim h^{2k} S(u)^2,
\end{align*}
which implies $\Vert { \nabla (u^e - u \circ \Phi_h) \cdot n } \Vert_{-\frac12,h,\Gamma_h} \lesssim h^{k+\frac12} S(u)$. Combining these estimates completes the proof.
\end{proof}
\begin{lemma} \label{lemapprox} Let $u$ be given with $u \in H^{3,\infty}(\Omega)$ if $k=2$, and   $u \in H^{k+1}(\Omega)$ if $k \geq 3$. The following holds:
\begin{equation*}
\inf_{v_h \in \mathcal{V}_h} \Vert u\circ \Phi_h - v_h \Vert_A + \Vert v_h \Vert_J   \lesssim h^k \begin{cases}
  \|u\|_{H^{3,\infty}(\Omega_1 \cup \Omega_2)} &~~\text{if}~~k=2, \\                                                                                
    \|u\|_{H^{k+1}(\Omega_1 \cup \Omega_2)} & ~~\text{if}~~k \geq 3. 
\end{cases}
\end{equation*}
\end{lemma}
\begin{proof}
Due to Lemma \ref{lemuext} it remains to derive a bound for
$$
\inf_{v_h \in \mathcal{V}_h} \Vert u^e - v_h \Vert_A + \Vert v_h \Vert_J.
$$
To this end we use unfitted interpolation strategies as in \cite{hansbo2002unfitted}. To the nodal interpolation operator $I_k$ in $\mathcal{V}_h^k$ we define the unfitted interpolation operator $v_h = (I_k u^e)|_{\Omega_h}$ to obtain the bounds (cf. \cite{LR16a} for details):
$$
| u^e\!\! - v_h |_{H^1(\Omega_h)} \lesssim h^k S(u),~~
\Vert u^e\!\! - v_h \Vert_{\frac12,h,\Gamma_h} \lesssim h^k S(u),~~
\Vert \partial_n (u^e\!\! - v_h) \Vert_{-\frac12,h,\Gamma_h} \lesssim h^k S(u).
$$
We note that Lemma \ref{lem:invtrafo} is crucial to obtain these optimal approximation results on the background mesh for the mapped finite element space $\mathcal{V}_h^k = V_h^k \circ \Theta_h^{-1}$, cf. \cite{ciarlet1972interpolation}. 
Finally we bound the ghost penalty contributions, $l = 1,..,k$:
\begin{align*}
  \sum_{F\in\mathcal{F}_{h}} \Vert \spacejump{\partial_n^l v_h} \Vert_{\Theta_h(F)}^2  
 & = \sum_{F\in\mathcal{F}_{h}} \Vert \spacejump{\partial_n^l (u^e - v_h)} \Vert_{\Theta_h(F)}^2  
 \lesssim \sum_{T \in\mathcal{T}_{h}} \Vert D^l (u^e - v_h) \Vert_{\partial \Theta_h(T)}^2  \\
 & \lesssim \sum_{T \in\mathcal{T}_{h}} h^{-1} \Vert D^{l} (u^e - v_h) \Vert_{\Theta_h(T)}^2  + h \Vert D^{l+1} (u^e - v_h) \Vert_{\Theta_h(T)}^2  \\
 & \lesssim \sum_{T \in\mathcal{T}_{h}} h^{2k-2l+1} \ S_T(u)^2,
\end{align*}
where $S_T(u)$ is the element localized version of $S(u)$ in \eqref{io}. 
Together with the scaling of the ghost penalty terms, we obtain the result.
\end{proof}

\subsection{A priori error bound}
As a direct consequence of the previous estimates we obtain the following a priori error bound.
\begin{corollary}
  Assume that $\lambda$ is sufficiently large and $h$ is sufficiently small, so that Lemma \ref{lem:gp} holds true.
  Let $u$ be the solution to \eqref{eq:ellmodel} with $u \in H^{3,\infty}(\Omega)$ if $k=2$, and $u \in H^{k+1}(\Omega)$ if $k \geq 3$. Further, assume $f \in H^{1,\infty}(\Omega)$, $u_D \in H^{1,\infty}(\Gamma)$ with the data extensions $f^e$ and $u_D^e$ fulfiling the requirements of Lemma \ref{lemconsist}. For $u_h \in \mathcal{V}_h$ the solution of \eqref{eq:Nitsche1} there holds
  \begin{equation}
    \Vert u \circ \Phi_h - u_h \Vert_A \lesssim h^k (S(u) + \Vert f \Vert_{1,\infty,\Omega} + \sqrt{h} \Vert u_D \Vert_{1,\infty,\Gamma}).
  \end{equation}
\end{corollary}

\begin{remark}
  We note that for every extension $w$ of $u$ with $\Vert w \circ \Phi_h - u^e \Vert_A \lesssim h^{k} S(u)$ there holds
  \begin{equation} \label{apriori2}
    \Vert w - u_h \Vert_A \lesssim h^k (S(u) + \Vert f \Vert_{1,\infty,\Omega} + \sqrt{h} \Vert u_D \Vert_{1,\infty,\Gamma}).
  \end{equation}
  For instance, for the extension $u^e$ in Lemma \ref{lemapprox} we have $\Vert u \circ \Phi_h - u^e \Vert \lesssim h^{k+\frac12} S(u)$, cf. \eqref{io}, and thus \eqref{apriori2} holds for $w=u^e$.
\end{remark}
\section{Numerical example} \label{sec:numex}
In this section we present results of two numerical experiments for the previously introduced and analyzed method. One example uses a uniform background mesh and homogeneous Dirichlet boundary data and the second one considers a mesh which is not quasi-uniform and has inhomogeneous Dirichlet data.
Both experiments were carried out in \texttt{ngsxfem} which is an add-on library to the finite element library \texttt{NGSolve} \cite{schoeberl2014cpp11}.

\subsection{Ring geometry on a uniform background mesh}\label{example1}

The problem has been investigated in \cite{boiveau2016fictitious}.
The background domain is $\widetilde\Omega = (-1,1)^2$ and the physical domain is a two-dimensional ring with inner radius $R_1=1/4$ and outer radius $R_2=3/4$, $\Omega = \{ \phi(x) \leq 0 \}$ with $\phi(x) := (r(x,y)-R_2)(r(x,y)-R_1),~ r(x,y) = \sqrt{x^2+y^2}$.
\begin{figure}[h!]
  \vspace*{-0.05cm}
  \begin{center}
    \includegraphics[width=0.35\textwidth,clip=true, trim=12mm 25mm 18mm 25mm]{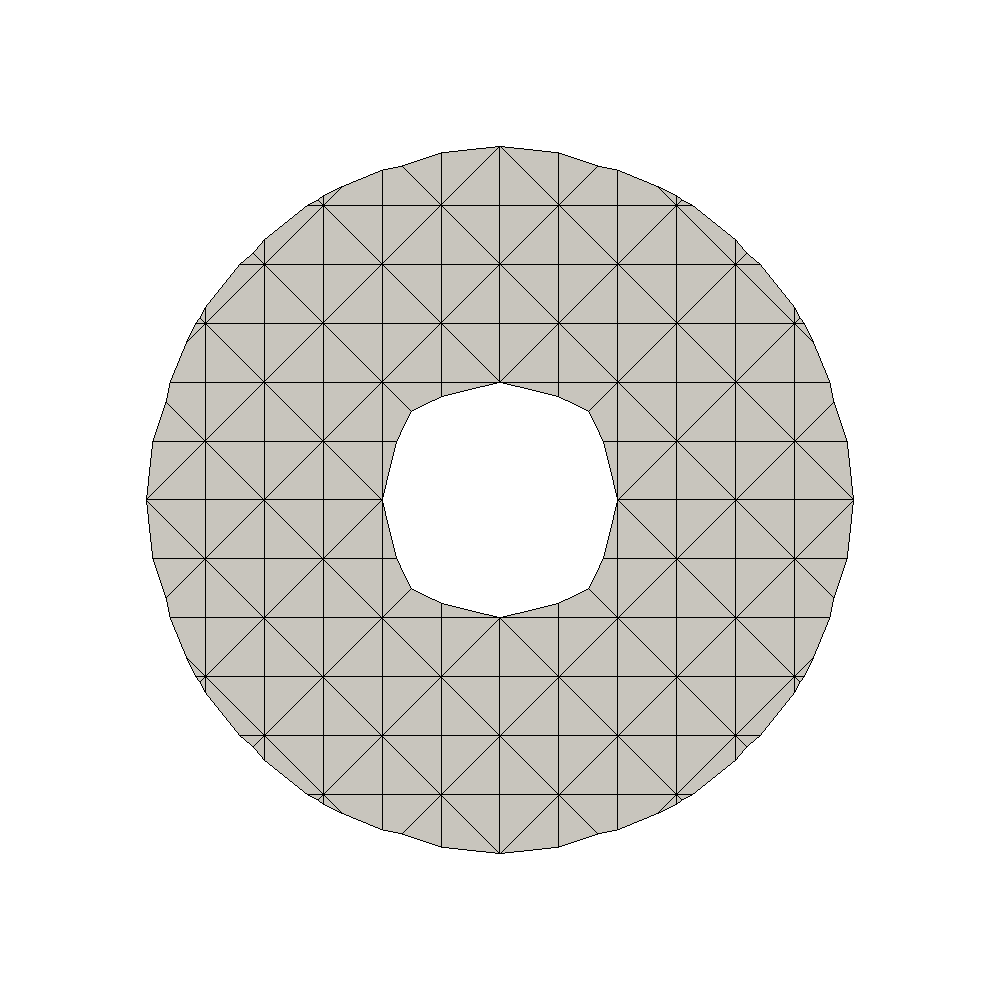} \hspace*{-0.04\textwidth}
    \includegraphics[width=0.35\textwidth,clip=true, trim=6mm 25mm 18mm 25mm]{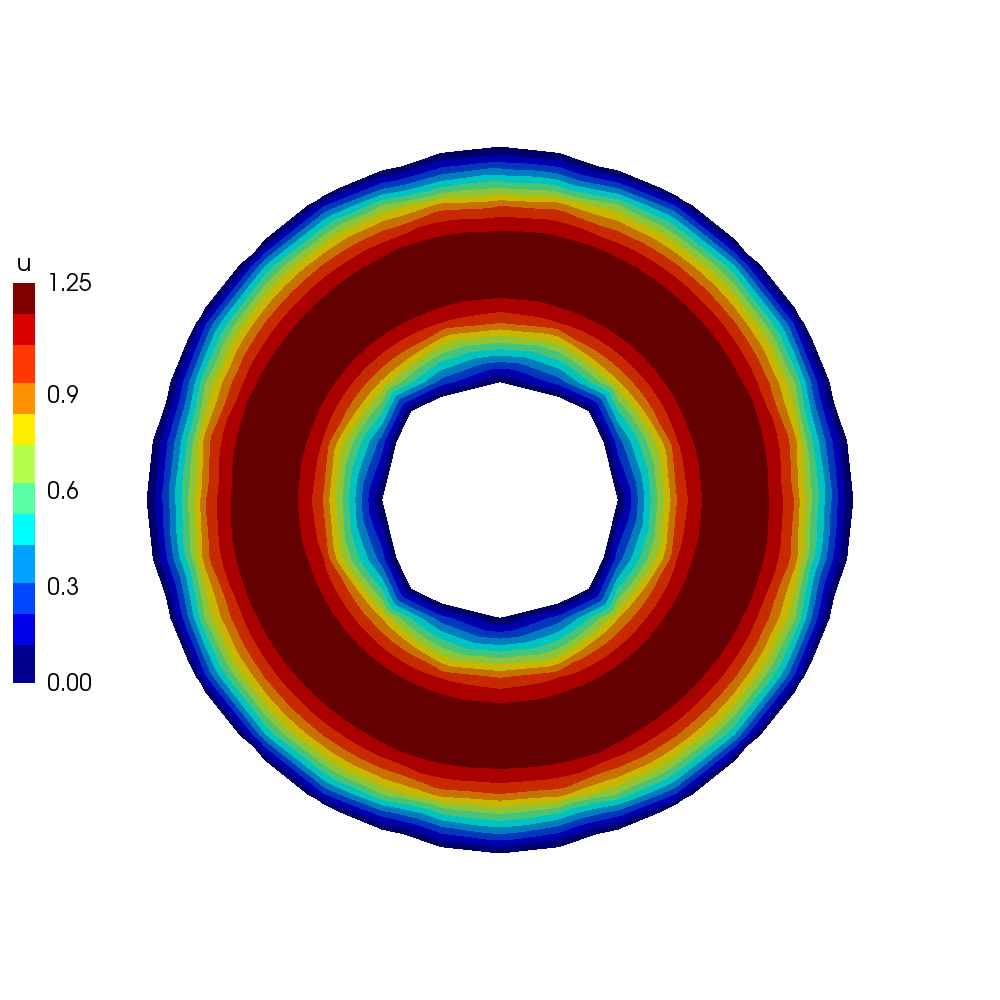} \hspace*{-0.04\textwidth} 
    \includegraphics[width=0.35\textwidth,clip=true, trim=6mm 25mm 18mm 25mm]{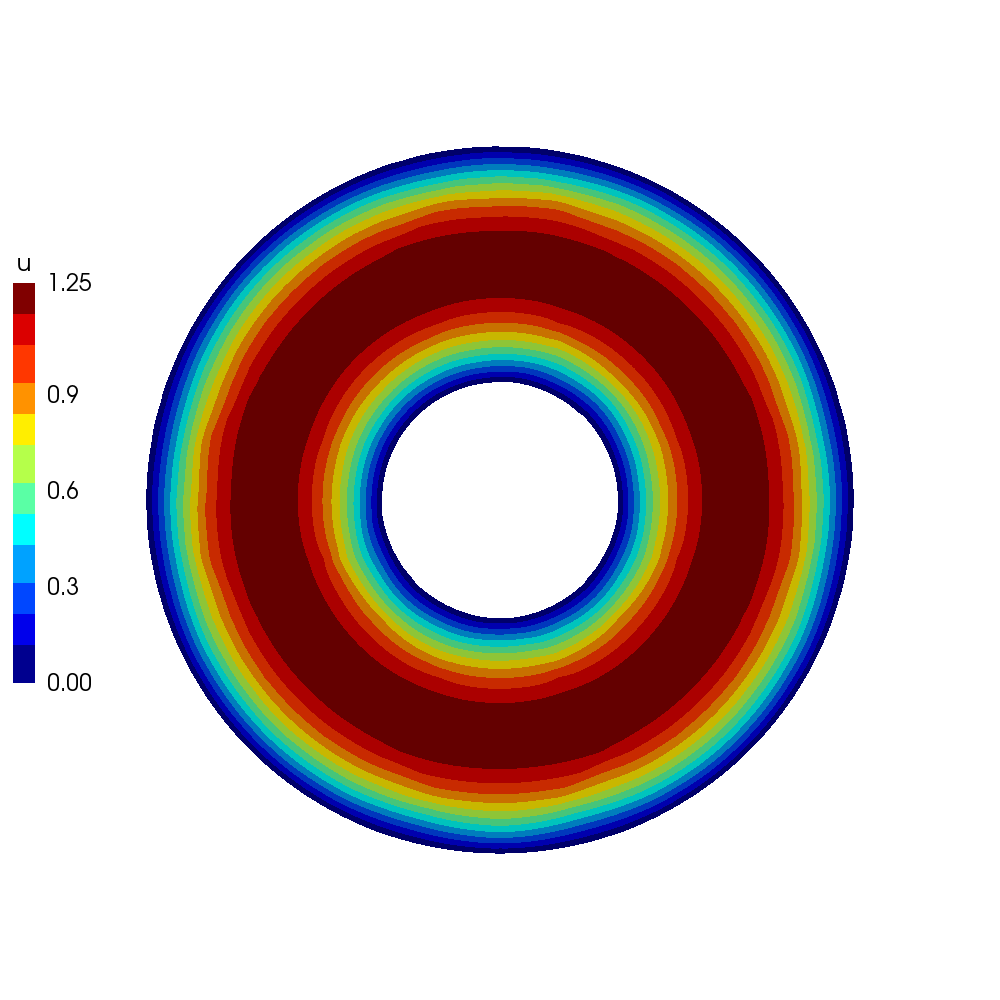} 
  \end{center}
  \vspace*{-0.7cm}
  \caption{Mesh and domain $\Omegalin$ on level $L=1$ (left), the discrete solution $u_h \circ \Theta_h \in V_h$ for $k=4$ on $\Omegalin$ (center) and the discrete solution $u_h \in \mathcal{V}_h$ for $k=4$ on $\Omega_h$ (right).} \label{fig:sketchsol}
\end{figure}
The level set function $\phi$ is approximated with $\phi_h \in V_h^k,~k=1,2,3,4$ by interpolation. We note that $\phi$ is not a signed distance function. Further, for $k=1$ we get $\Theta_h = \operatorname{id}$, i.e. the mesh is unchanged.
For the problem in \eqref{eq:ellmodel}, we take $u_D=u_D^e = 0$ and right-hand side $f$ such that the solution is given by
$
u(x) = 20(3/4-r(x))(r(x)-1/4), \ x \in \Omega$.
For the numerical evaluation of errors we use a function $u^e$.
We choose canonical extensions of $f^e$ and $u^e$ by evaluting the formulas for $f$ and $u$ also in $\Omega_h$.
We start with a regular structured ``criss-cross'' mesh of size $8 \times 8$ and denote this as mesh level $L=0$. Mesh levels $L>0$ are obtained by repeatedly applying uniform mesh refinements.
In Fig. \ref{fig:sketchsol} the mesh on level $L=1$ and the domain $\Omegalin$ is displayed alongside with the numerical solution for $k=4$ (on $\Omegalin$ and $\Omega_h$).

\begin{figure}
  \vspace*{-0.2cm}
  \begin{center}
    \includegraphics[width=0.48\textwidth, trim=2mm 2mm 2mm 0mm]{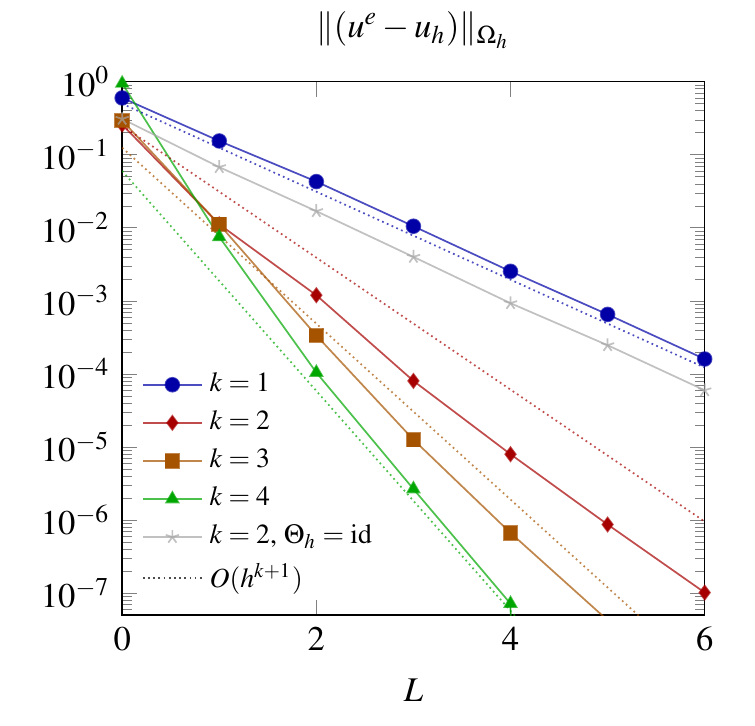} \hspace*{-0.025\textwidth}
    \includegraphics[width=0.48\textwidth, trim=2mm 2mm 2mm 0mm]{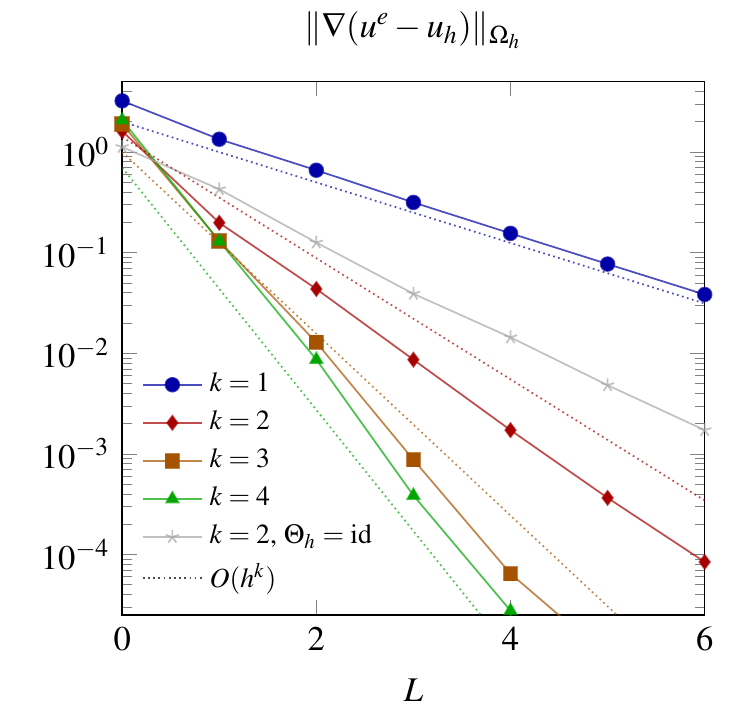}
  \end{center}
  \vspace*{-0.3cm}
   \caption{Convergence behavior under mesh refinement in the $L^2(\Omega_h)$ and the $H^1(\Omega_h)$-semi-norm.}
  \label{fig:numex1}
\end{figure}

\begin{figure}[h!]
  \vspace*{-0.2cm}
  \begin{center}
    \includegraphics[width=0.48\textwidth, trim=2mm 2mm 2mm 0mm]{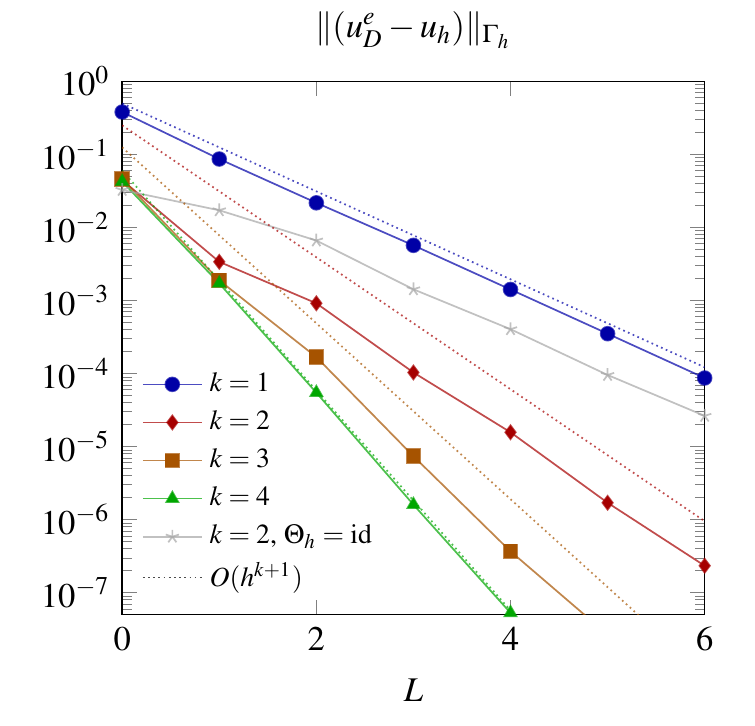} \hspace*{-0.025\textwidth}
  \end{center}
  \vspace*{-0.3cm}
  \caption{Convergence behavior under mesh refinement in the $L^2(\Gamma_h)$-semi-norm.}
  \label{fig:numex2}
\end{figure}

We choose the Nitsche parameter $\lambda = 10 k^2$ and the ghost penalty parameter $\gamma_l = 0.2 k (l-1)!^{-2}$. The scaling with $k$ in the Nitsche parameter stems from the dependency of the constant in the trace inverse estimate on $k$. The scaling of the ghost penalty paramter is motivated by the scaling of the derivative terms in the Taylor expansion in the proof of Lemma \ref{lem:gp}, cf. \eqref{gpest}.

In Fig. \ref{fig:numex1} we observe the convergence behavior of the isoparametric finite element method for norms in $\Omega_h$. We make the following observation. The $O(h^k)$ error bound for the $H^1$-norm as predicted by our error analysis is sharp. In the $L^2$-norm we also observe the optimal $\mathcal{O}(h^{k+1})$ convergence of the error.
The approximation of boundary values, i.e. the $L^2(\Gamma_h)$ norm of the error is depicted in Fig. \ref{fig:numex2} where we observe an $\mathcal{O}(h^{k+1})$ behavior although the previous error analysis only predicted the slightly worse bound $\mathcal{O}(h^{k+\frac12})$.

In both figure we also added the results of a discretization with $k=2$ and $\Theta_h=\operatorname{id}$, i.e. a higher order discretization on the low order geometry approximation $\Omega_h=\Omegalin$, as a comparison. In further experiments we found that the results for $k=3,4$ (and $\Omega_h=\Omegalin$) are almost identical. From this comparison, we easily conclude that the higher order geometry approximation in this method is crucial to obtain higher order convergence.

\subsection{Ellipsoid on a mesh that is not quasi-uniform}\label{example2}
As a second example we consider a background mesh which is not quasi-uniform. The background domain is $\tilde{\Omega}= (-1,1) \times (-1.1,1.1)$. The physical domain is an ellipsoid with half axes coinciding with the $x$ and $y$ direction and lengths $2/3$ ($x$-direction) and $2$ ($y$-direction). The corresponding level set function is $\phi(x) = \sqrt{3 x^2 + y^2} - 1$ so that $\Omega = \{\phi(x) < 0 \}$. As an initial mesh we consider a uniform mesh that is three times adaptively refined on all elements which have a non-trivial intersection with $\{ (x,y) \in \Omega \text{ and } |y| > 0.8\}$, see also Figure \ref{fig:numexb2} where the initial mesh is shown.
We set the right hand side data to $f=\cos(y)$ and choose $u_D =\cos(y)$ so that $u(x) = cos(y)$ solves the Poisson problem on $\Omega$. Again, we use the canonical extensions to define $u_D^e$, $f^e$ and $u^e$.
The stabilization parameters for the ghost penalty and the Nitsche stabilization parameter are chosen as before in section \ref{example1}. Starting from the initially non-quasi-uniform mesh we apply successive uniform refinements compute numerical approximations with the isoparametric unfitted method for order $k=1,2,3$ and measure the same errors as before in section \ref{example1}.

We note that the quasi-uniformity assumption is clearly not provided by this series of meshes in this example. However, we observe the same optimal convergence rates as in the example with the uniform background mesh.

\begin{figure}[h!]
  \vspace*{-0.05cm}
  \begin{center}
    \includegraphics[width=0.48\textwidth,clip=true, trim=5mm 5mm 5mm 5mm]{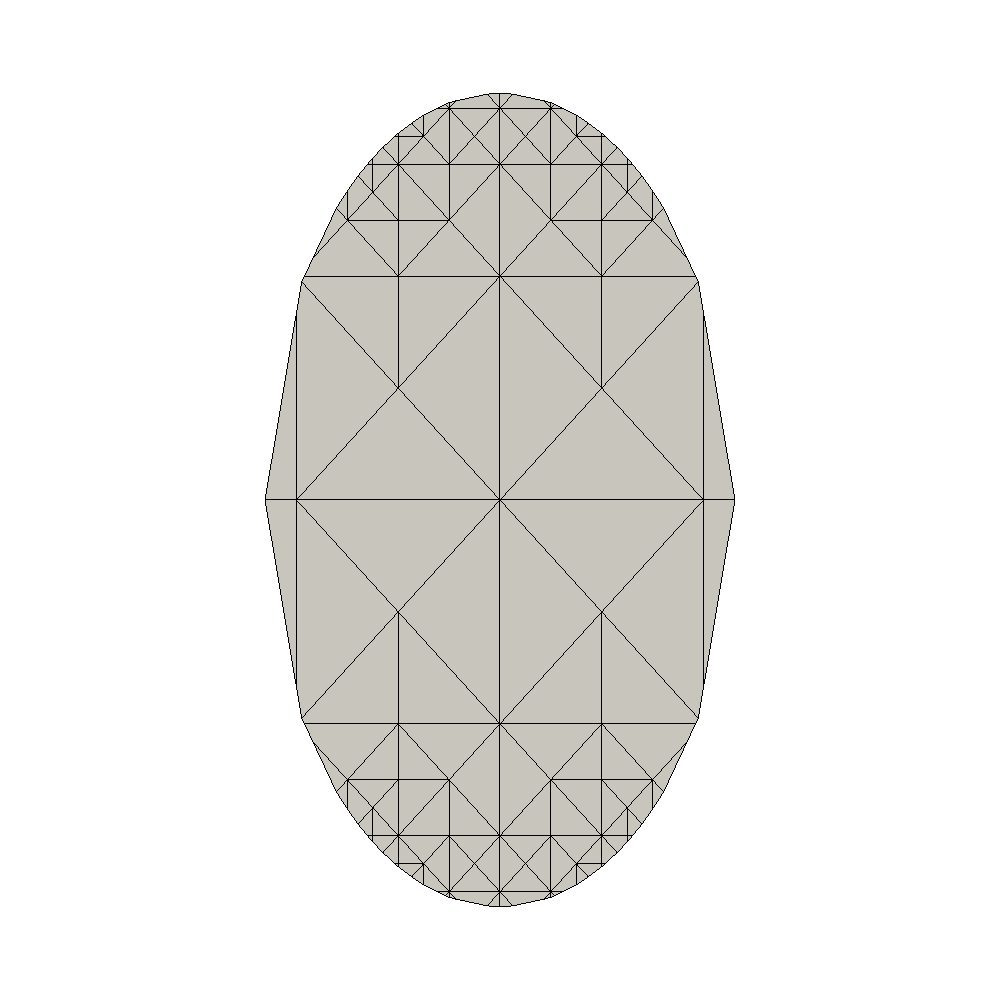}
    \includegraphics[width=0.48\textwidth, trim=2mm 2mm 2mm 0mm]{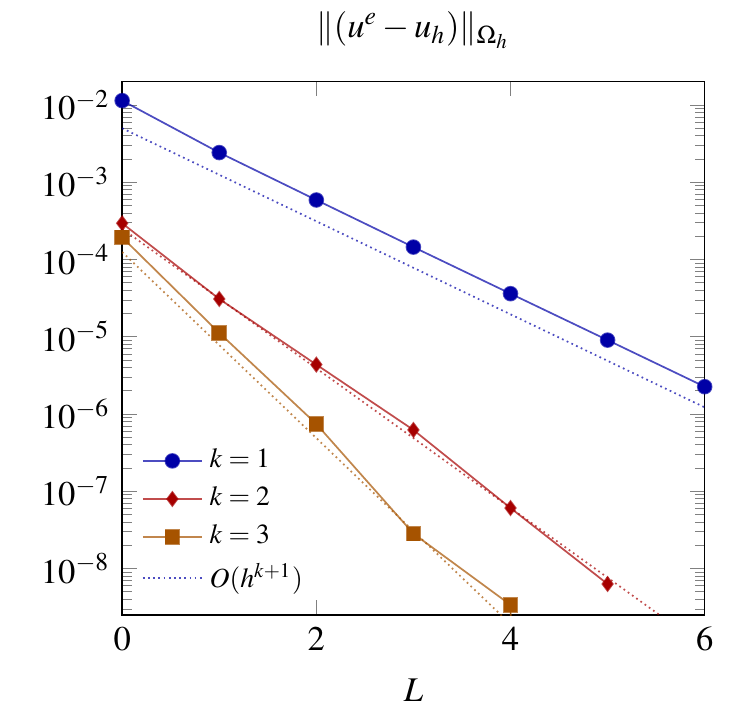}   \end{center}
  \vspace*{-0.4cm}
  \caption{Mesh and domain $\Omegalin$ to example in section \ref{example2} on level $L=0$ (left), convergence behavior under (uniform) mesh refinement in the $L^2(\Omega_h)$ (right).} \label{fig:numexb1}
\end{figure}

\begin{figure}
  \vspace*{-0.2cm}
  \begin{center}
    \includegraphics[width=0.48\textwidth, trim=2mm 2mm 2mm 0mm]{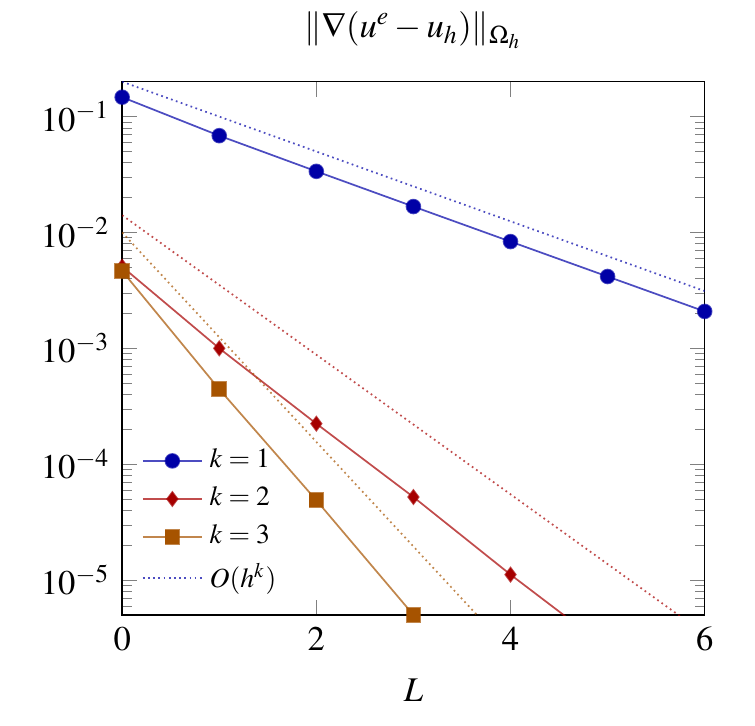} \hspace*{-0.025\textwidth}
    \includegraphics[width=0.48\textwidth, trim=2mm 2mm 2mm 0mm]{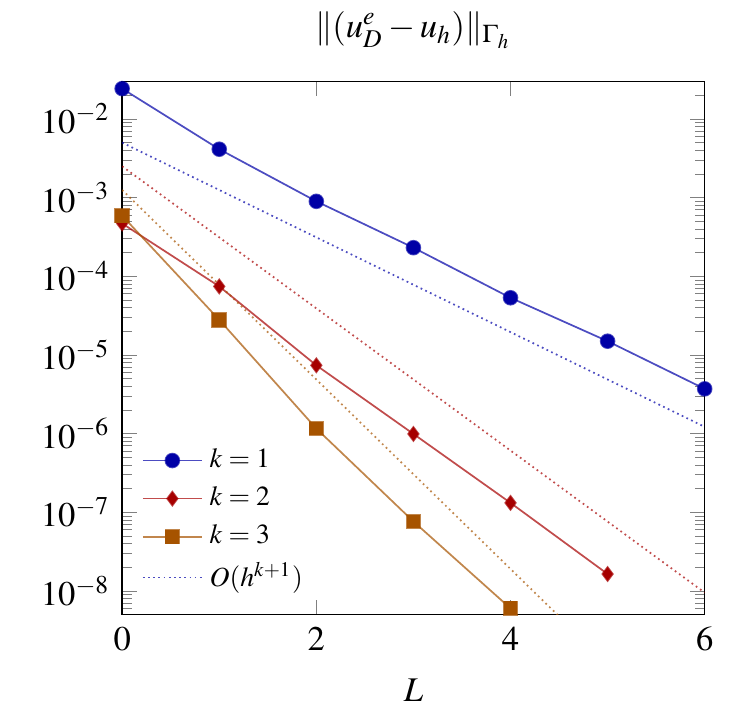}
  \end{center}
  \vspace*{-0.3cm}
   \caption{Convergence behavior under mesh refinement in the $H^1(\Omega_h)$-semi-norm and .}
  \label{fig:numexb2}
\end{figure}

\section{Conclusion and outlook} \label{sec:outlook}
We presented the concept of isoparametric unfitted finite element methods. To this end we discussed the construction of the isoparametric mapping $\Theta_h$ and its most important properties. Based on this mesh transformation we formulated a higher order isoparametric fictitious domain finite element method in the spirit of the stabilized Nitsche discretization in \cite{burman2012fictitious}. The main contribution in this paper is the a priori error analysis. For this, we made use of recent results from previous works on isoparametric unfitted finite element methods, cf. \cite{GLR16,LPWL_PAMM_2016,LR16a,LR16b}. However, the interplay of curved meshes and the higher order ghost penalty stabilization required new components in the error analysis to establish the coercivity result in Lemma \ref{lem:gp}. In a numerical experiment we validated the predictions arising from the a priori error analysis.

The methodology of isoparametric unfitted finite element methods offers one possibility to deal with the geometry approximation of implicitly described geometries with higher order accuracy. The ghost penalty method on the other hand offers a mechanism to robustly handle small cut configurations. One important feature of ghost penalty stabilizations is the fact that also the conditioning of linear systems is independent of the cut geometry.
However, this comes at the price of adding more terms (depending on the order of the discretization) to the variational formulation. Further, the ghost penalty integrals introduce additional couplings in the linear systems. More importantly, it is unclear how to efficiently precondition linear systems arising from (higher order) ghost penalty discretizations. This is a topic that is independent of the quality of the geometry approximation and should be investigated in the future.

\begin{acknowledgement}
The author gratefully acknowledges funding by the German Science Foundation (DFG) within the project ``LE 3726/1-1''.
\end{acknowledgement}

\appendix

\section{Appendix: Selected proofs}\label{sec:proofs}
We give the proofs of some of the more technical results in section \ref{sec:analysis}. 
\subsubsection*{Proof of Lemma \ref{lem:invtrafo}}
  The first estimate has been proven in \cite{GLR16} for $T \in \mathcal{T}_h^\Gamma$. With the extension operator applied in the projection step $P_h^2$ this property carries over to every element $T \in \mathcal{T}_h$, cf. the analysis in \cite{LR16a}.
  The proof for the inverse transformation is based around the following estimate from \cite{ciarlet1972interpolation}:
  \begin{align}
    \vert \Theta_h^{-1} \vert_{l,\infty,\Theta_h(T)}
    & \lesssim
      | \Theta_h^{-1} |_{1,\infty,\Theta_h(T)}
     \sum_{m=2}^l | \Theta_h |_{m,\infty,T} \sum_{j \in \mathcal{I}(m,l)} \prod_{n=1}^{l-1} \vert \Theta_h^{-1} \vert_{n,\infty,\Theta_h(T)}^{j_n} \nonumber \\ 
    \text{ with } \quad \mathcal{I}(m,l) & := \{ j = (j_1,..,j_{l-1}) \mid \sum_{n=1}^{l-1} j_n = m, \sum_{n=1}^{l-1} n j_n = l\}. \nonumber
  \end{align}
  Starting with $\vert \Theta_h^{-1} \vert_{1,\infty,\Theta_h(T)} \lesssim 1$ which follows from Lemma \ref{propertiesdh} the claim follows by induction.

\subsubsection*{Proof of Lemma \ref{lemhighder}}
  We define $\hat{v} := v \circ \Theta_h$ with $\hat{v}|_T \in \mathcal{P}^k(T),~ T \in \mathcal{T}_h^{\Gamma,+}$.
  There holds the following estimate due to a higher order chain rule for multivariate functions, cf. \cite{ciarlet1972interpolation},
  \begin{align}
    \vert \hat{v} \circ \Theta_h^{-1} \vert_{l,\infty,\Theta_h(T)}
   & \lesssim
     \sum_{m=1}^l | \hat{v} |_{m,\infty,T} \sum_{j \in \mathcal{J}(m,l)} \prod_{n=1}^l \vert \Theta_h^{-1} \vert_{n,\infty,\Theta_h(T)}^{j_n} \nonumber \\ 
    \text{ with } \quad \mathcal{J}(m,l) & := \{ j = (j_1,..,j_l) \mid \sum_{n=1}^l j_n = m, \sum_{n=1}^l n j_n = l\}. \nonumber
  \end{align}
  There holds the finite element inverse inequality
  $
  \vert \hat{v} \vert_{j,\infty,T} \lesssim h^{-j} \Vert \hat{v} \Vert_{\infty,T},~j\geq 0.
  $
  Now, with Lemma \ref{lem:invtrafo} we have
  $ \vert \Theta_h^{-1} \vert_{j,\infty,\Theta_h(T)} \lesssim 1$ and $\vert \Theta_h \vert_{j,\infty,T} \lesssim 1$ and $D^{k+1} \hat{v} = 0$
  which completes the proof.

\subsubsection*{Proof of Lemma \ref{lem:gp}}
  We show \eqref{eq:a}. The proof of estimate \eqref{eq:b} follows similar lines.
  
  We mimic the proof of \cite[Theorem 5.1]{MassingLarsonLoggEtAl2013a}, but need a few more technical steps due to $F$ being curved and $D^{k+1}v \neq 0$.
  First, we introduce simply connected domains $B_1 \subset T_1$, $C_2 \subset T_2$ and $F^\ast \subset F$ with $\operatorname{diam}(B_1),\operatorname{diam}(C_2),\operatorname{diam}(F^\ast) \gtrsim h$. For such domains standard finite element estimates give (with $\hat{v} = v \circ \Theta_h^{-1}$)
  $$
  \Vert  v \Vert_{T_1}^2 \simeq \Vert  \hat{v} \Vert_{\hat{T}_1}^2 \simeq  \Vert  \hat{v} \Vert_{\Theta_h^{-1}(B_1)}^2 \simeq   \Vert  v \Vert_{B_1}^2,
  $$
  and a similar result for $C_2$.
  For a simply connected domain $F^\ast \subset F$ we define $T_i^\ast(F^\ast) := \{ x \in T_i \mid x = x_F + \gamma n_F(x_F), x_F \in F, \gamma \in \rr \}$. For $x = x_F + \gamma n_F(x_F)$ in $T_i^\ast$ we define the mirror point $M(x) = x_F - \gamma n_F(x_F)$.
  Now, for $h$ sufficiently small we find domains $F^\ast \subset F$, a ball $B_1 \subset T_1^\ast(F^\ast)$ and $C_2 := M(B_1) = \{ x = M(y), y \in B_1\}$ which fulfil the aforementioned requirements.

  To each point $x_1= x_F + \gamma n_F(x_F)$ in $B_1$ we have a corresponding point
  $x_2 = M(x_1)$ in $C_2$. We develop $v_i:=v|_{T_i}$ around $x_F$ and obtain (for a $\xi_i = x_F \pm \gamma_{\xi} n_F \in T_i,~\gamma_{\xi} \in [0,\gamma]$)
  
  $$
   v_i(x_i) =  v_i(x_F) 
  + \sum_{l=1}^{k} \frac{\gamma^l}{l!}  \frac{\partial^l  v_i}{\partial n^l}(x_F) 
  + \frac{\gamma_{\xi}^{k+1}}{(k+1)!} \frac{\partial^{k+1} v_i}{\partial n^{k+1}}(\xi).
  $$
  Subtracting and integrating over $B_1$ then gives
\begin{equation} \label{gpest}
  \Vert  v_1 \Vert_{B_1}^2
\lesssim \Vert  v_2\circ M \Vert_{B_1}^2 + k \sum_{l=1}^k \frac{h^{2l+1}}{l!^2}  \Vert \spacejump{\partial_n^l v} \Vert_F^2 + 2 |B_1| \frac{h^{2k+2}}{(k+1)!^2} \Vert D^{k+1} v \Vert_{\infty,T_1 \cup T_2}^2.
\end{equation}
Exploiting the properties of $M$, and Lemma \ref{lemhighder}, we get
\begin{align*}
  \Vert  v \Vert_{T_1}^2 &
 \leq c \Vert  v \Vert_{T_2}^2 + J_F(v,v) + c h^2 \left( \Vert v \Vert_{T_1}^2 + \Vert v \Vert_{T_2}^2 \right).
\end{align*}
Now, for $h$ sufficiently small the last term can be absorbed by the others and the claim holds true.

\bibliographystyle{siam}
\bibliography{literature}
\end{document}